\documentclass[a4paper,11pt]{article}

\usepackage[a4paper]{geometry}

\usepackage{authblk}

\usepackage{amsmath}
\usepackage{amsthm}
\usepackage{amssymb}
\usepackage{amsfonts}
\usepackage{mathtools}
\usepackage{bm}
\usepackage{bussproofs}

\usepackage{cancel}
\usepackage{lscape}

\usepackage{float}
\usepackage{booktabs}
\usepackage{caption}
\usepackage{adjustbox}

\usepackage{graphicx}
\usepackage{tikz}
\usetikzlibrary{cd,backgrounds,positioning,trees,shapes,arrows,patterns,topaths,calc}

\usepackage[
  backend=biber,
  natbib=true,
  sorting=nyt,
  maxcitenames=9,
  maxbibnames=99,
  giveninits=true,
  abbreviate=false,
  eprint=false,
  doi=true,
]{biblatex}
\DeclareFieldFormat{doi}{\url{https://doi.org/#1}}
\DeclareFieldFormat{url}{\url{#1}}
\usepackage{hyperref}
\hypersetup{
  hidelinks,
  linktoc = all,
  breaklinks = true,
}
\usepackage{aliascnt}
\usepackage[nameinlink]{cleveref}

\newtheorem{theorem}{Theorem}
\numberwithin{theorem}{section}

\newaliascnt{lemma}{theorem}
\newtheorem{lemma}[lemma]{Lemma}
\aliascntresetthe{lemma}

\newaliascnt{corollary}{theorem}
\newtheorem{corollary}[corollary]{Corollary}
\aliascntresetthe{corollary}

\newaliascnt{fact}{theorem}

\aliascntresetthe{fact}

\newaliascnt{question}{theorem}

\aliascntresetthe{question}

\newaliascnt{proposition}{theorem}
\newtheorem{proposition}[proposition]{Proposition}
\aliascntresetthe{proposition}

\Crefname{theorem}{Theorem}{Theorems}
\Crefname{lemma}{Lemma}{Lemmas}
\Crefname{corollary}{Corollary}{Corollaries}
\Crefname{fact}{Fact}{Facts}
\Crefname{question}{Question}{Questions}
\Crefname{proposition}{Proposition}{Propositions}

\theoremstyle{definition}

\newaliascnt{definition}{theorem}
\newtheorem{definition}[definition]{Definition}
\aliascntresetthe{definition}

\newaliascnt{remark}{theorem}
\newtheorem{remark}[remark]{Remark}
\aliascntresetthe{remark}

\newaliascnt{example}{theorem}
\newtheorem{example}[example]{Example}
\aliascntresetthe{example}

\newaliascnt{notation}{theorem}
\newtheorem{notation}[notation]{Notation}
\aliascntresetthe{notation}

\Crefname{definition}{Definition}{Definitions}
\Crefname{remark}{Remark}{Remarks}
\Crefname{example}{Example}{Examples}
\Crefname{notation}{Notation}{Notations}

\theoremstyle{remark}

\newaliascnt{claim}{theorem}
\newtheorem{claim}[claim]{Claim}
\aliascntresetthe{claim}

\Crefname{claim}{Claim}{Claims}

\renewcommand{\>}{\rangle}

\renewcommand{\sf}[1]{\mathsf{#1}}

\renewcommand{\frak}[1]{\mathfrak{#1}}

\newcommand{\alg}[1]{\mathfrak{#1}}
    \newcommand{\A}{\alg{A}}
    \newcommand{\B}{\alg{B}}

\renewcommand{\sp}[1]{\mathfrak{#1}}
    \newcommand{\X}{\sp{X}}
    \newcommand{\Y}{\sp{Y}}

\newcommand{\fr}[1]{\mathcal{#1}}
    \newcommand{\F}{\fr{F}}
    \newcommand{\G}{\fr{G}}

\newcommand{\class}[1]{\mathcal{#1}}
    \newcommand{\V}{\class{V}}
    \newcommand{\U}{\class{U}}
    
    \newcommand{\D}{\class{D}}
    
    \newcommand{\C}{\class{C}}
    \newcommand{\R}{\class{R}}
    \renewcommand{\S}{\class{S}}
    
    \newcommand{\FH}{\sf{FH}}
    \newcommand{\KF}{\sf{KF}}

\newcommand{\fin}{_\mathrm{fin}}
\newcommand{\si}{_\mathrm{si}}

\renewcommand{\sup}{\mathrm{sup}}

\renewcommand{\r}[1]{\mathrm{rank} (#1)}

\newcommand{\pow}[1]{\mathcal{P} (#1)}
\newcommand{\Log}{\mathsf{Log}}

\newcommand{\logic}[1]{\mathsf{#1}}

    \newcommand{\K}{\logic{K}}
    \newcommand{\Kf}{\logic{K4}}
    \newcommand{\Kff}[2]{\logic{K4^{#1}_{#2}}}
    \newcommand{\Sf}{\logic{S4}}
    \newcommand{\wKf}{\logic{wK4}}
    \newcommand{\IPC}{\logic{IPC}}

\newcommand{\NExt}[1]{\mathop{\mathsf{NExt}}{#1}}

\renewcommand{\sf}[1]{\mathsf{#1}}

\newcommand{\Clop}{\mathsf{Clop}}

\renewcommand{\phi}{\varphi}
\newcommand{\emp}{\emptyset}

\newcommand{\conti}{2^{\aleph_0}}

\newcommand{\Dia}{\Diamond}

\newcommand{\Boxx}[1]{\Box^{#1}}

\newcommand{\inj}{\hookrightarrow}
\newcommand{\surj}{\twoheadrightarrow}

\newcommand{\Lor}{\bigvee}
\newcommand{\Land}{\bigwedge}

\def\irrefsingle{\tikz[baseline=.1ex]{
\fill (0,0.6ex) circle (2pt) coordinate (A);}
}

\def\refseeirref{\tikz[baseline=.1ex]{
\draw(0,0.6ex) circle (2pt) coordinate (A);
\fill(3ex,0.6ex) circle (2pt) coordinate (B);
\draw[->, shorten <=0.4ex, shorten >=0.4ex] (A)-> (B);}
}

\newenvironment{acknowledgements}{%
  \begin{abstract}
}{%
  \end{abstract}
}

\renewcommand{\S}{\mathcal{S}}

\renewcommand{\FH}{\mathcal{FH}}

\begin{document}

\title{Chopping More Finely: Finite Countermodels in \\ Modal Logic via the Subdivision Construction}
\author{Tenyo Takahashi\footnote{\href{mailto:t.takahashi@uva.nl}{t.takahashi@uva.nl}}}
\affil{\small Institute for Logic, Language and Computation, \\ University of Amsterdam}
\date{}
\maketitle

\begin{abstract}
We present a new method, the Subdivision Construction, for proving the finite model property (the fmp) for broad classes of modal logics and modal rule systems. The construction builds on the framework of stable canonical rules, and produces a finite modal space, dually, a finite modal algebra, that serves as a finite countermodel of such rules, yielding the fmp. We apply the Subdivision Construction to prove the fmp for logics and rule systems axiomatized by stable canonical formulas and rules of finite modal algebras of finite height. As a consequence, we identify a class of union-splittings in $\NExt{\Kf}$ with degree of Kripke incompleteness 1.\footnote{This paper is based on \cite[Chapter 4]{TakahashiThesis}.}
\end{abstract}

\section{Introduction}

A normal modal logic $L$ has the \emph{finite model property} (the \emph{fmp}, for short) if every formula that is not a theorem of $L$ is refuted by a finite countermodel of $L$. Algebraically, this is equivalent to saying that the variety of modal algebras validating $L$ is generated by its finite members. The fmp is one of the central and well-studied properties of modal logics (see, e.g., \cite{blackburnModalLogic2001,czModalLogic1997}). The fmp strictly implies Kripke completeness; a Kripke complete logic without the fmp can be found in \cite{gabbayDecidableFinitelyAxiomatizable1971}. Moreover, every finitely axiomatizable logic with the fmp is decidable \cite{Harrop1958}.

Many general fmp results for classes of modal logics are known; see \cite{czModalLogic1997,blackburnModalLogic2001} for more examples than the ones discussed in this paragraph. Among the early results, Bull \cite{bullThatAllNormal1966} and Fine \cite{fineLogicsContaining431971} showed that every extension of $\logic{S4.3}$ has the fmp, and Segerberg \cite{Segerberg1971} proved that every extension of $\Kf$ of finite depth has the fmp. Later work shifted to identifying classes of logics with specific semantic properties that have the fmp, in other words, to finding sufficient conditions for having the fmp. \emph{Subframe logics} in $\NExt{\Kf}$ were defined and studied by Fine \cite{fineK4I, fineK4II}, and Zakharyaschev \cite{zakharyaschevCanonicalFormulas1, CanonicalFormulasK421996} generalized them to \emph{cofinal subframe logics}; all transitive subframe logics and cofinal subframe logics have the fmp. Moreover, Zakharyaschev applied two combinatorial methods based on \emph{canonical formulas}, the method of \emph{inserting points} \cite{zakharyaschevSufficientConditionFinite1993} and the method of \emph{removing points} \cite{ZakharyaschevCanonical3}, for proving the fmp for large classes of transitive logics (see also \cite[Chapter 11]{czModalLogic1997}). These results imply, in particular, the fmp of extensions of $\Kf$ with modal reduction principles and the fmp of extensions of $\Sf$ with a formula in one variable. More recently, Bezhanishvili et al. \cite{stablemodallogic} introduced \emph{$\logic{M}$-stable logics} and showed that they have the fmp when the logic $\logic{M}$ \emph{admits filtration in the strong sense}. This assumption can be weakened to $\logic{M}$ admitting \emph{definable filtration} \cite{takahashi2025stablecanonicalrulesformulas}. Another important fmp result comes from the study of lattices of modal logics. Blok \cite{blok1978degree} showed that all \emph{union-splittings} in $\NExt{\K}$ have the fmp. The fmp is also studied for modal rule systems, or modal consequence relations, of multi-conclusion rules, namely universal theories of modal algebras. It was shown in \cite{stablecanonicalrules} that all \emph{stable rule systems} have the fmp.

In this paper, we introduce a new combinatorial method, which we call the \emph{Subdivision Construction}, to show the fmp for modal logics and modal rule systems. The construction is based on the theory of \emph{stable canonical rules} introduced in \cite{stablecanonicalrules}. Every formula $\phi$ is semantically equivalent to finitely many stable canonical rules, say, $\{\rho(\A_i, D_i): 1 \leq i \leq n\}$, where each $\A_i$ is a finite $L$-algebra and $D_i \subseteq A_i$. So, if $\phi \notin L$, then some $L$-algebra $\B$ refutes $\phi$, and hence $\B \not\models \rho(\A_i, D_i)$ for some $i$. By the semantic characterization of stable canonical rules, this means that there is a stable embedding from $\A_i$ into $\B$ satisfying the closed domain condition (CDC) for $D_i$ (we recall these notions in \Cref{Sec 2}). Let $\X$ and $\F_i$ be the dual spaces of $\B$ and $\A_i$, respectively. The stable embedding above dualizes to a surjective stable map from $\X$ onto $\F_i$ satisfying CDC for a set $\D_i$ of clopen subsets of $\F_i$. If we can find a finite modal space $\F'$ such that $\F' \models L$ with a surjective stable map from $\F'$ onto $\F_i$ satisfying CDC for $\D_i$, then the dual algebra of $\F'$ is a finite $L$-algebra refuting $\rho(\A_i, D_i)$, and thus refuting $\phi$. This yields the fmp of $L$. In this way, we can reduce a refutation of an arbitrary formula to a refutation of a stable canonical rule, and then to a combinatorial problem on modal spaces. This strategy was used, for example, in \cite{bezhanishviliBLOKESAKIATHEOREMSSTABLE2025a} to provide an alternative proof of the Blok-Esakia theorem \cite{blok1976,esakia1976}. 

The Subdivision Construction starts from a stable map $f$ from $\X$ onto $\F_i$, and constructs a finite modal space $\F'$ by refining the finite quotient represented by $f$. Since $\F_i$ is finite, the map $f$ divides $\X$ into finitely many regions, namely the preimages $f^{-1}(w)$ for $w \in \F_i$. The construction makes this division finer by subdividing each preimage into finitely many subregions. This motivates the name \emph{Subdivision Construction}. 
The construction proceeds point by point. For a point $w \in \F_i$, we divide the preimage $f^{-1}(w)$ into finitely many subregions, replace the point $w$ by names for these subregions, and obtain a finite modal space $\F'_w$. The construction also induces a map from $\X$ to $\F'_w$, which agrees with $f$ except for $f^{-1}(w)$ and sends each point in a subregion to the corresponding name, together with a map from $\F'_w$ to $\F_i$ that identifies all these names with $w$. This is how the construction works at the point $w$.

\begin{figure}[H]
    \centering
    \begin{tikzpicture}[scale=1]

    \node at (0, 2.5) {$\X$};
    \draw (0,0) circle [x radius = 1.5, y radius = 1.8];
    \draw (0.2,0.3) circle [x radius = 0.8, y radius = 1];
    
    \draw (-0.1,-0.64) -- (-0.1,1.24);
    \draw (0.45,-0.64) -- (0.45,1.24);

    \draw (-0.58,0.6) -- (0.98,0.6);
    \draw (-0.58,0) -- (0.98,0);

    \node at (6, 2.5) {$\F_i$};
    \draw[->>] (1,2.5) -- (5,2.5);
    \node at (3, 2.8) {$f$};
    \draw (6,0) circle [x radius = 1.2, y radius = 1.5];
    \node (1) at (6.2,0.4) {$\bullet$};
    \node at (6.5,0.4) {$w$};
    
    \draw[-] (0.22,1.31) -- (6.2,0.45);
    \draw[-] (0.22,-0.71) -- (6.2,0.35);

    \draw (9,0) circle [x radius = 1.2, y radius = 1.5];
    \node (1) at (8.8,0.9) {$\bullet$};
    \node (1) at (9.2,0.9) {$\bullet$};
    \node (1) at (9.6,0.9) {$\bullet$};
    \node (1) at (8.8,0.4) {$\bullet$};
    \node (1) at (9.2,0.4) {$\bullet$};
    \node (1) at (9.6,0.4) {$\bullet$};
    \node (1) at (8.8,-0.1) {$\bullet$};
    \node (1) at (9.2,-0.1) {$\bullet$};
    \node (1) at (9.6,-0.1) {$\bullet$};
\end{tikzpicture}
\caption{The Subdivision Construction at $w$.}
\end{figure}

Let the \emph{rank} of a point in $\F_i$ be the length of the longest path starting at that point. In practice, the construction starts with the points of rank $1$, namely the dead ends in $\F_i$, and proceeds inductively according to rank. Points of infinite rank are left unchanged. After finitely many steps, we obtain a finite modal space $\F'$, a surjective map $f':\X \to \F'$, and a surjective map $g:\F' \to \F_i$ such that the following diagram commutes.

\begin{figure}[H]
        \centering
        \adjustbox{scale=1,center}{
\begin{tikzcd}
\X \arrow[rr, "f"] \arrow[rd, "f'"'] &                      & \F_i \\
                                     & \F' \arrow[ru, "g"'] &   
\end{tikzcd}
        }
\end{figure}

The precise properties of $\F'$, $f'$, and $g$ are stated in \Cref{4: Lem main}. It turns out that in certain cases, as will be demonstrated in \Cref{Sec 4}, we can show that $\F' \models L$. It follows that $L$ has the fmp, as discussed in the previous paragraph. This general scheme of applying the Subdivision Construction to establish the fmp is formulated as \Cref{4: Thm main scheme}.

We apply the Subdivision Construction to prove the fmp for the following classes. These fmp results, involving stable logics and stable canonical formulas, parallel the general fmp result regarding cofinal subframe logics and canonical formulas in \cite[Theorem 11.55]{czModalLogic1997}, originally proved by Zakharyaschev \cite{ZakharyaschevCanonical3} using the aforementioned method of removing points. Both results establish the fmp for classes of logics of the form ``a special logic + some special formulas.'' Recall that a modal algebra is of finite height if it validates $\Box^{n+1} \bot$ for some $n \in \omega$.
\begin{enumerate}
    \item A stable rule system + stable canonical rules of finite modal algebras of finite height.
    \item A strictly stable logic + stable canonical formulas of finite s.i.~modal algebras of finite height, where a logic $L$ is \emph{strictly stable} if the class of s.i.~$L$-algebras is closed under finite stable subalgebras.
    \item A $\Kff{m+1}{1}$-stable logic + stable canonical formulas of finite s.i.~$\Kff{m+1}{1}$-algebras of finite height (for $m \geq 1$), where $\Kff{m+1}{1}$ is the logic $\K + \Dia^{m+1} p \to \Dia p$. 
\end{enumerate}
Items (1) and (3) substantially generalize the fmp results for stable rule systems \cite{stablecanonicalrules} and $\Kff{m+1}{1}$-stable logics \cite{takahashi2025stablecanonicalrulesformulas} (see also \cite{stablemodallogic} for the case of $\Kf$-stable logics) by allowing additional axioms. 

We also analyze these classes from the viewpoint of lattices of modal logics and rule systems. Recall that a normal extension $L$ of a logic $L_0$ is a \emph{splitting} in $\NExt{L_0}$ if there is a logic $L' \in \NExt{L_0}$ such that, for every $L'' \in \NExt{L_0}$, exactly one of $L \subseteq L''$ and $L'' \subseteq L'$ holds. A logic $L \in \NExt{L_0}$ is a \emph{union-splitting} if it is the join of a set of splittings in $\NExt{L_0}$. The same definitions apply to rule systems. We observe that the logics and rule systems listed above are union-splittings in the corresponding lattices. In particular, in item (2), if we take the base logic to be $\K$ and extra axioms to be Jankov formulas of finite s.i.~modal algebras of finite height (see Subsection~\ref{subsec 2.3}), then we obtain the fmp of all union-splittings in $\NExt{\K}$, originally proved by Blok \cite{blok1978degree}.

Item (3) has an additional consequence for degrees of Kripke incompleteness. Recall that a logic $L$ has \emph{degree of Kripke incompleteness} 1 in $\NExt{L_0}$ if, for every logic $L' \in \NExt{L_0}$ with $L \neq L'$, some Kripke frame validates one of $L$ and $L'$ and refutes the other \cite{fineIncompleteLogicContainingS41974} (see also \cite[Section 10.5]{czModalLogic1997}). Blok's dichotomy theorem gives a complete characterization of the degrees of Kripke incompleteness in $\NExt{\K}$ \cite{blok1978degree}: union-splittings have degree 1, while other logics have degree $\conti$. For $\NExt{\Kf}$, however, the problem remains open \cite[Problem 10.5]{czModalLogic1997}, and the situation is far less clear. Fine's incomplete logic in $\NExt{\Kf}$ \cite{fineIncompleteLogicContainingS41974} shows that there is a logic with degree at least 2 in $\NExt{\Kf}$, but it is not known whether any logic has degree at least 3. On the degree 1 side, the main known result is that logics of finite depth have degree 1 in $\NExt{\Kf}$, which follows from the fact that they are union-splittings in $\NExt{\Kf}$ (see, e.g., \cite{FrameBasedFormulas2008}) and have the fmp \cite{Segerberg1971}. The fmp result in (3) above yields classes of union-splittings in $\NExt{\Kff{m+1}{1}}$ that have degree $1$ in the corresponding lattices. In the case of $\Kf$, we will observe that the logics with degree 1 obtained in this way are not of finite depth. Thus, the Subdivision Construction produces a new class of logics with degree 1 in $\NExt{\Kf}$, disjoint from the finite-depth class, and contributes to the problem of determining degrees of Kripke incompleteness in $\NExt{\Kf}$.

The paper is organized as follows. In \Cref{Sec 2}, we recall the required preliminaries on modal algebras, modal spaces, stable canonical rules, and stable canonical formulas. \Cref{Sec 3} develops the Subdivision Construction for modal spaces, proving \Cref{4: Lem main}. In \Cref{Sec 4}, we apply \Cref{4: Lem main} through the general scheme of \Cref{4: Thm main scheme} to prove the fmp results described above. We also observe that the logics and rule systems addressed are union-splittings in the corresponding lattices, and derive the consequence on degrees of Kripke incompleteness. Finally, in \Cref{Sec 5}, we discuss possible further applications of the Subdivision Construction.

\section{Preliminaries} \label{Sec 2}

We assume familiarity with basic modal logic and refer to \cite{blackburnModalLogic2001, czModalLogic1997} for syntax and Kripke semantics. We will use $\Box^m \phi$ as an abbreviation of $\Box \cdots \Box \phi$ with $m$ many $\Box$ and $\Boxx{\leq m}$ as an abbreviation of $\phi \land \cdots \land \Box^m \phi$; the dual abbreviation applies to $\Dia$. We only work with normal modal logics in this paper, so we call them simply \emph{logics}. For a logic $L$, let $\NExt{L}$ denote the lattice of normal extensions of $L$. Recall that pre-transitive logics $\Kff{m}{n}$ are logics axiomatized by $\Dia^m p \to \Dia^n p$ (or equivalently, $\Box^n p \to \Box^m p$) over $\K$. The logic $\Kff{m}{n}$ defines the condition $\forall x \forall y \: (x R^m y \to x R^n y)$, where $R^k$ is the $k$-time composition of $R$, on modal spaces. Note that $\Kff{2}{1}$ is the transitive logic $\Kf$.

\subsection{Modal algebras and modal spaces}

Recall that a \emph{modal algebra} is a pair $\A = (A, \Dia)$ of a Boolean algebra $A$ and a unary operation $\Dia$ on $A$ such that $\Dia 0 = 0$ and $\Dia (a \lor b) = \Dia a \lor \Dia b$. Valuations, satisfaction, and validity are defined as usual. We call a modal algebra $\A$ an $L$-algebra for a logic $L$ if $\A \models L$. The above abbreviation applies to operations on modal algebras as well. Recall that a \emph{modal space} is a pair $\X = (X, R)$ of a Stone space $X$ and a binary relation $R \subseteq X \times X$ such that $R[x]$ is closed for each $x \in X$ and the set $R^{-1}[U]$ is clopen for each clopen set $U \subseteq X$. A subset $U \subseteq X$ is an \emph{upset} if $x \in U$ and $xRy$ implies $y \in U$. Modal spaces are also known as \emph{descriptive frames}. Valuations, satisfaction, and validity are defined as usual. We call a modal space $\X$ an $L$-space for a logic $L$ if $\X \models L$. Finite modal spaces are identified with finite Kripke frames with the discrete topology. Continuous p-morphisms between modal spaces are simply called p-morphisms. It is well-known that the category of modal algebras and homomorphisms is dually equivalent to the category of modal spaces and p-morphisms (see, e.g., \cite{venema6AlgebrasCoalgebras2007}). 

We sketch the construction of the dual equivalence functors, which is built on top of \emph{Stone duality}. Given a modal algebra $\A$, the dual modal space of $\A$ is $\A_* = (A_*, R_{\Dia})$ where $A_*$ is the dual Stone space of $A$, and $R_{\Dia}$ is defined by $x R_{\Dia} y$ iff for any $a \in A$, $a \in y$ implies $\Dia a \in x$. Recall that the topology of $\A_*$ is generated by the base $\{\sigma(a): a \in A\}$ where $\sigma(a) = \{x \in A_*: a \in x\}$. For a homomorphism $h: \A \to \B$, the dual p-morphism of $h$ is $h_*: \B_* \to \A_*; x \mapsto h^{-1}[x]$. Conversely, given a modal space $\X = (X, R)$, the dual modal algebra of $\X$ is $\X^* = (X^*,  \Dia_R)$ where $X^* = \Clop(X)$, the Boolean algebra of clopen subsets of $X$, and $\Dia_R \: a = R^{-1}[a]$ for $a \in X^*$. For a p-morphism $f: \X \to \Y$, the dual homomorphism of $f$ is $f^*: \Y^* \to \X^*; a \mapsto f^{-1}[a]$. Under this duality, subalgebras correspond to p-morphic images and homomorphic images to closed upsets. 

We assume familiarity with basic universal algebra and refer to \cite{UniversalAlgebraFundamentals2011, ACourseInUniversalAlgebra1981} for an introduction. For a class $\class{K}$ of modal algebras, $\V(\class{K})$ (resp. $\U(\class{K})$) denotes the least variety (resp. universal class) containing $\class{K}$. Recall that logics correspond one-to-one to varieties by the operations $\V(L) \coloneq \{\A: \A \models L\}$ for a logic $L$ and $\Log(\V) \coloneq \{\phi: \V \models \phi\}$ for a variety $\V$. We will use the following dual characterization of subdirectly irreducible (s.i., for short) modal algebras by Venema \cite{venemaDualCharSI2004}. A point $x$ in a modal space $\X$ is called a \emph{topo-root} of $\X$ if the closure of $R^{< \omega}[x]$ is $X$. A modal space $\X$ is called \emph{topo-rooted} if the set of topo-roots of $\X$ has a non-empty interior. 

\begin{theorem} \label{4: Thm si toporooted}
    A modal algebra is s.i. iff its dual space is topo-rooted. 
\end{theorem}

Recall that $x \in \X$ is a \emph{root} of $\X$ if $R^{< \omega}[x] = X$ and $\X$ is \emph{rooted} if $\X$ has a root. In a finite modal space, the notions of topo-roots and roots coincide because every subset is clopen. So, a finite modal space is topo-rooted iff it is rooted.

\emph{Stable homomorphisms} and the \emph{closed domain condition} for modal algebras and modal spaces were introduced in \cite{stablecanonicalrules}, generalizing that for Heyting algebras and Priestley space in \cite{LocallyFiniteReducts2017}. Intuitively, a stable homomorphism $h$ does not preserve $\Dia$ (so it is not a modal algebra homomorphism), but it does satisfy the inequality $\Dia h(a) \leq h(\Dia a)$; the closed domain condition indicates for which elements $a$ the equality $\Dia h(a) = h(\Dia a)$ holds. Stable homomorphisms were studied in \cite{TopocanonicalCompletionsClosure2008} as semihomomorphisms and in \cite{Ghilardi01012010} as continuous morphisms. 

\begin{definition}
    Let $\A$ and $\B$ be modal algebras. 
    \begin{enumerate}
        \item A Boolean homomorphism $h: A \to B$ is a \emph{stable homomorphism} if $\Dia h(a) \leq h(\Dia a)$ for all $a \in A$.
        \item Let $h: A \to B$ be a stable homomorphism. For $a \in A$, we say that $h$ satisfies the \emph{closed domain condition (CDC) for $a$} if $h(\Dia a) = \Dia h(a)$. For $D \subseteq A$, we say that $h$ satisfies the \emph{closed domain condition (CDC) for $D$} if $h$ satisfies CDC for all $a \in D$.
    \end{enumerate}
\end{definition}

These definitions dualize to modal spaces as follows.

\begin{definition}
    Let $\X = (X, R)$ and $\Y = (Y, Q)$ be modal spaces. 
    \begin{enumerate}
        \item A continuous map $f: X \to Y$ is a \emph{stable map} if $xRy$ implies $f(x)Qf(y)$ for all $x, y \in X$.
        \item Let $f: X \to Y$ be a stable map. For a clopen subset $D \subseteq Y$, we say that $f$ satisfies the \emph{closed domain condition (CDC) for $D$} if
\[
Q[f(x)] \cap D \neq \emptyset \Rightarrow f(R[x]) \cap D \neq \emptyset.
\]
    For a set $\D$ of clopen subsets of $Y$, we say that $f: X \to Y$ satisfies \emph{the closed domain condition (CDC) for $\D$} if $f$ satisfies CDC for all $D \in \D$.
    \end{enumerate}
    
\end{definition}

It follows directly from the definition that a stable homomorphism $h: A \to B$ satisfying CDC for $A$ is a modal algebra homomorphism, and a stable map $f: X \to Y$ satisfying CDC for $\Clop(Y)$ is a p-morphism.

\begin{notation} \leavevmode
    \begin{enumerate}
        \item We write $\A \inj \B$ if $\A$ is a subalgebra of $\B$ and $\B \surj \A$ if $\A$ is a homomorphic image of $\B$. We write $\Y \inj \X$ if $\Y$ is a closed upset of $\X$ and $\X \surj \Y$ if $\Y$ is a p-morphic image of $\X$.
        \item We write $h: \A \inj_D \B$ if $h$ is a stable embedding satisfying CDC for $D$, and $\A \inj_D \B$ if there is such an $h$. We write $f: \X \surj_\D \Y$ if $f$ is a surjective stable map satisfying CDC for $\D$, and $\X \surj_\D \Y$ if there is such an $f$.
    \end{enumerate}
\end{notation}

Note that if $\A \inj_D \B$, then $\A \inj_{D'} \B$ for any $D' \subseteq D$, and similarly for stable maps between modal spaces. We call $\A$ a \emph{stable subalgebra} of $\B$ if $\A \inj_\emp \B$, and $\Y$ a \emph{stable image} of $\X$ if $\X \surj_\emp \Y$.

\subsection{Definable filtrations}

A logic $L$ has the \emph{finite model property} (the \emph{fmp}, for short)
if the variety $\V(L)$ is generated by its finite members. That is, for any formula $\phi \notin L$, there is a finite $L$-algebra (or, equivalently, a finite $L$-space) that refutes $\phi$. The fmp is often shown via the method of \emph{filtration}. We refer to \cite[Definition 2.36]{blackburnModalLogic2001} and \cite[Section 5.3]{czModalLogic1997} for the standard filtration for Kripke frames and models, and \cite[Section 4]{stablecanonicalrules} for an algebraic presentation. Recall that in standard filtration, the equivalence relation $\sim_\Theta$ is determined by the subformula-closed set $\Theta$. However, one can use a finer relation as long as the number of equivalence classes remains finite. The idea of \emph{definable filtration} is to use a finer equivalence relation that is induced by a set of formulas, so that the projection map from a modal space to its filtration is continuous, which enables an algebraic presentation. While the idea of definable filtration appeared first in \cite{ageneralfiltrationmethod1972}, it was explicitly introduced for Kripke frames in \cite{Kikot2020}, and an algebraic formulation was provided in \cite{takahashi2025stablecanonicalrulesformulas}. 

\begin{definition} \label{3: Def alg filtration}
    Let $\A = (A, \Dia)$ be a modal algebra, $V$ be a valuation on $\A$, $\Theta$ be a finite subformula-closed set of formulas, and $\Theta'$ be a finite subformula-closed set of formulas containing $\Theta$. A \emph{definable filtration of $(\A, V)$ for $\Theta$ through $\Theta'$} is a modal algebra $\A' = (A', \Dia')$ with a valuation $V'$ such that:
    \begin{enumerate}
        \item $A'$ is the Boolean subalgebra of $A$ generated by $V[\Theta'] \subseteq A$, 
        \item $V'(p) = V(p)$ for $p \in \Theta'$ and $V'(p) = 0$ for $p \notin \Theta'$,
        \item The inclusion $\A' \inj \A$ is a stable homomorphism satisfying CDC for $D$, where 
        \[D = \{V(\phi) : \Dia \phi \in \Theta\}.\]
    \end{enumerate}
    We also call $\A'$ a \emph{definable filtration of $\A$ for $\Theta$ through $\Theta'$}.
\end{definition}

Note that the extended set $\Theta'$ of formulas is used to generate the Boolean subalgebra $A'$, while the closed domain $D$ is determined by $\Theta$. A standard filtration through $\Theta$ is a special type of definable filtration where $\Theta' = \Theta$. We can automatically deduce the fmp for logics and rule systems that \emph{admit definable filtration}. 

\begin{definition} \label{3: Def admit filtration} \leavevmode
    \begin{enumerate}
        \item A class $\C$ of modal algebras \emph{admits definable filtration} if for any finite subformula-closed set $\Theta$ of formulas, there is a finite subformula-closed set $\Theta'$ containing $\Theta$ such that, for any modal algebra $\A \in \C$ and any valuation $V$ on $\A$, there is a definable filtration $(\A', V')$ of $(\A, V)$ for $\Theta$ through $\Theta'$ such that $\A' \in \C$.
        \item  A modal logic $L$ \emph{admits definable filtration} if the variety $\V(L)$ admits definable filtration. 
        \item A rule system $\S$ \emph{admits definable filtration} if the universal class $\U(\S)$ admits definable filtration. 
    \end{enumerate}
\end{definition}

\begin{theorem} \label{3: Thm filtration implies fmp}
    If a logic $L$ admits definable filtration, then it has the fmp. If a rule system $\S$ admits definable filtration, then it has the fmp.
\end{theorem}

Note that a logic $L$ admits definable filtration iff the rule system $\S_L$ admits definable filtration because they correspond to the same class of modal algebras. Many logics are known to admit standard filtration and thus admit definable filtration. For example, $\K$, $\logic{T}$, and $\logic{D}$ admit the least and the greatest filtration. The transitive logics $\Kf$ and $\Sf$ admit the Lemmon filtration (see \cite[Section 5.3]{czModalLogic1997} or \cite[Section 2.3]{blackburnModalLogic2001}). A definable filtration for the pre-transitive logics $\Kff{m+1}{1}$ $(m \geq 1)$ was constructed in \cite{ageneralfiltrationmethod1972} (see \cite{takahashi2025stablecanonicalrulesformulas} for an algebraic presentation).

\subsection{Stable canonical rules and formulas} \label{subsec 2.3}

We refer to \cite{stablecanonicalrules, Jeřábek_2009, kracht8ModalConsequence2007} for inference rules and consequence relations in modal logic. We will only address modal multi-conclusion rules and normal modal multi-conclusion consequence relations, or normal modal multi-conclusion rule systems, so we simply call them \emph{rules} and \emph{rule systems}. For a set $\R$ of rules, let $\S_0 + \R$ be the least rule system containing $\S_0 \cup \R$. For a logic $L$, let $\S_L$ be the least rule system containing $L$. When $\S = \S_0 + \R$, we say that $\S$ is \emph{axiomatized} by $\R$ over $\S_0$. When $\S_L$ is axiomatized by $\R$ over $\S_{L_0}$, we say that $L$ is axiomatized by $\R$ over $L_0$. Two rules $\rho$ and $\rho'$ are \emph{equivalent} over $\S$ if $\S + \rho = \S + \rho'$. Each rule corresponds to a universal sentence, and rule systems correspond one-to-one to universal classes by the operations $\U(\S) \coloneq \{\A: \A \models \S\}$ for a rule system $\S$ and $\S(\U) \coloneq \{\rho: \U \models \rho\}$ for a universal class $\U$. For a rule system $\S$, let $\NExt{\S}$ denote the lattice of normal extensions of $\S$. The notion of fmp applies to rule systems as well: a rule system $\S$ has the fmp if the universal class $\U(\S)$ is generated by its finite members. That is, for any rule $\rho \notin \S$, there is a finite $\S$-algebra (or, equivalently, a finite $\S$-space) that refutes $\rho$.

\emph{Stable canonical rules} are introduced in \cite{stablecanonicalrules} as an alternative to canonical rules, the theory of which is developed in \cite{Jeřábek_2009}. The basic idea of stable canonical rules, as well as other characteristic formulas, is analogous to that of \emph{diagrams} widely used in model theory: to encode the structure of finite algebras or finite spaces (frames), but only partially. They are defined from finite algebras and finite spaces, and their validity has a purely semantic characterization.

\begin{definition} \label{3: Def scr}
    Let $\A$ be a finite modal algebra and $D \subseteq A$. The \emph{stable canonical rule} $\rho(\A, D)$ associated to $\A$ and $D$ is the rule $\Gamma/\Delta$, where:
    \begin{align*}
        \Gamma  = & \{p_a \lor p_b \leftrightarrow p_{a \lor b} : a,b \in A\} \cup \\
        & \{\lnot p_a \leftrightarrow p_{\lnot a} : a \in A\} \cup \\
        & \{\Dia p_a \rightarrow p_{\Dia a} : a \in A\} \cup \\
        & \{p_{\Dia a} \rightarrow \Dia p_a : a \in D\},
    \end{align*}
    and
    \[\Delta = \{p_a : a \in A, a \ne 1\}.\]
\end{definition}

\begin{theorem} \label{3: Thm sc rule char}
    Let $\A$ be a finite modal algebra, $D \subseteq A$, and $\B$ be a modal algebra. Then 
    \[\B \not\models \rho(\A, D) \text{ iff } \A \inj_D \B.\]
\end{theorem}

We also write $\rho(\F, \D)$ for $\rho(\A, D)$ if $\A$ is dual to $\F$ and $\D = \sigma(D)$; see \cite[Remark 5.8]{stablecanonicalrules} for a direct definition via finite modal spaces. Then, \Cref{3: Thm sc rule char} dualizes as follows. For a finite modal space $\F$, $\D \subseteq \pow{F}$, and a modal space $\X$, we have 
\[\X \not\models \rho(\F, \D) \text{ iff } \X \surj_\D \F.\]

Over a rule system that admits definable filtration, every rule is semantically equivalent to finitely many stable canonical rules \cite{takahashi2025stablecanonicalrulesformulas}. Note that we can replace $\S$ with a logic $L$ by taking $\S = \S_L$ and $\rho$ with a formula $\phi$ by taking $\rho = /\phi$.

\begin{theorem} \label{3: Thm scr complete} \leavevmode
    Let $\S$ be a rule system that admits definable filtration. For any rule $\rho$, there exist stable canonical rules $\rho(\A_1, D_1), \dots, \rho(\A_n, D_n)$ where each $\A_i$ is a finite $\S$-algebra and $D_i \subseteq A_i$, such that for any $\S$-algebra $\B$,
    \[\B \models \rho \text{ iff } \B \models\rho(\A_1, D_1), \dots, \rho(\A_n, D_n).\]  
\end{theorem}

\emph{Stable canonical formulas} were introduced for $\Kf$ in \cite{stablecanonicalrules} and generalized to the pre-transitive logics $\Kff{m+1}{1}$ $(m \geq 1)$ in \cite{takahashi2025stablecanonicalrulesformulas} as a type of \emph{characteristic formulas} (also called \emph{algebra-based formulas} and \emph{frame-based formulas}). We refer to \cite{JankovFormulasAxiomatization2022b} for a comprehensive overview of the techniques of characterization formulas for superintuitionistic logics. 

\begin{definition} \label{3: Def scf for pretran}
    Let $\A$ be a finite s.i.~$\Kff{m+1}{1}$-algebra and $D \subseteq A$. Let $\rho(\A, D) = \Gamma / \Delta$ be the stable canonical rule defined in \Cref{3: Def scr}. We define the \emph{stable canonical formula} $\gamma^m(\A, D)$ as 
    \[\gamma^m(\A, D) = \bigwedge\{\Box^{\leq m} \gamma: \gamma \in \Gamma\} \to \bigvee\{\Box^{\leq m}\delta: \delta \in \Delta\}.\]
\end{definition}

We write $\gamma(\A, D)$ for $\gamma^1(\A, D)$. Similar to stable canonical rules, stable canonical formulas have the following semantic characterization and axiomatize all logics in $\NExt{\Kff{m+1}{1}}$ \cite{takahashi2025stablecanonicalrulesformulas}. 

\begin{theorem} \label{3: Thm scf char}
    Let $\A$ be a finite s.i.~$\Kff{m+1}{1}$-algebra and $D \subseteq A$. Then, for any $\Kff{m+1}{1}$-algebra $\B$, 
    \[\B \not\models \gamma^m(\A, D) \text{ iff there is a s.i.~homomorphic image $\C$ of $\B$ such that $\A \inj_D \C$}.\]
\end{theorem}

We also write $\gamma^m(\F, \D)$ for $\gamma^m(\A, D)$ if $\A$ is dual to $\F$ and $\D = \sigma(D)$. Then, \Cref{3: Thm scf char} dualizes as follows. For a finite rooted $\Kff{m+1}{1}$-space $\F$, $\D \subseteq \pow{F}$, and a $\Kff{m+1}{1}$-space $\X$, we have 
\[\X \not\models \gamma^m(\F, \D) \text{ iff there is a closed upset $\Y$ of $\X$ such that $\Y \surj_\D \F$}.\]

\begin{theorem} \label{3: Thm scf complete}
    Let $m \geq 1$. Every logic $L \supseteq \Kff{m+1}{1}$ is axiomatizable over $\Kff{m+1}{1}$ by stable canonical formulas. Moreover, if $L$ is finitely axiomatizable over $\Kff{m+1}{1}$, then $L$ is axiomatizable over $\Kff{m+1}{1}$ by finitely many stable canonical formulas.
\end{theorem}

We can define stable canonical formulas for $\K$ with similar effects as those for $\Kff{m+1}{1}$, but only for finite s.i.~modal algebras of finite height. A modal algebra $\A$ is of \emph{height $\leq n$} if $\A \models \Box^{n+1}\bot$, or equivalently, $\Box^{n+1} 0_\A = 1_\A$; $\A$ is of \emph{finite height} if it is of height $\leq n$ for some $n \in \omega$.

\begin{definition}
    Let $\A$ be a finite s.i.~modal algebra of height $\leq n$ and $D \subseteq A$. Let $\rho(\A, D) = \Gamma / \Delta$ be the stable canonical rule defined in \Cref{3: Def scr}. We define the \emph{stable canonical formula} $\epsilon(\A, D)$ as
    \[\epsilon(\A, D) = (\Box^{n+1} \bot \land \Land \{\Boxx{\leq n} \gamma : \gamma \in \Gamma\}) \to \Lor \{ \Boxx{\leq n} \delta : \delta \in \Delta\}.\]
\end{definition}

\begin{theorem} \label{3: Thm splitting fml char}
    Let $\A$ be a finite s.i.~modal algebra of finite height and $D \subseteq A$. Then, for any modal algebra $\B$,
    \[\B \not\models \epsilon(\A, D) \text{ iff there is a s.i.~homomorphic image $\C$ of $\B$ such that $\A \inj_D \C$}.\]
\end{theorem}

Two special types of stable canonical rules and formulas arise when we take the closed domain $D$ to be $\emp$ or the entire set $A$. We call a stable canonical rule of the form $\rho(\A, \emp)$ a \emph{stable rule} and a stable canonical formula of the form $\gamma^m(\A, \emp)$ a \emph{stable formula}. \emph{Stable rule systems} and \emph{stable logics} were introduced and studied in \cite{stablecanonicalrules,stablemodallogic} (see also \cite{FiltrationRevisitedLattices2018} for a comprehensive introduction).

\begin{definition} \label{3: Def stable}
    Let $\logic{M}$ be a logic.
    \begin{enumerate}
        \item A class $\class{K}$ of $\logic{M}$-algebras is \emph{$\logic{M}$-stable} if for any $\logic{M}$-algebra $\A$ and any $\B \in \class{K}$, if $\A \inj_\emp \B$, then $\A \in \class{K}$. The class $\class{K}$ is \emph{finitely $\logic{M}$-stable} if for any finite $\logic{M}$-algebra $\A$ and any $\B \in \class{K}$, if $\A \inj_\emp \B$, then $\A \in \class{K}$. We drop the prefix ``$\logic{M}$-'' when $\logic{M} = \K$.
        \item A rule system $\S$ is \emph{stable} if the universal class $\U(\S)$ is stable.
        \item A logic $L \supseteq \logic{M}$ is \emph{$\logic{M}$-stable} if $\V(L)$ is generated by an $\logic{M}$-stable class. We drop the prefix ``$\logic{M}$-'' when $\logic{M} = \K$.
    \end{enumerate}
\end{definition}

The following theorems are respectively from \cite{stablecanonicalrules} and \cite{takahashi2025stablecanonicalrulesformulas} (the case $\Kf$ is proved in \cite{stablemodallogic}).

\begin{theorem} \label{4: Thm stable rule system and logic and stable rule}
    A rule system $\S$ is stable iff $\S$ is axiomatizable over $\S_\K$ by stable rules. A logic $L$ is stable iff $L$ is axiomatizable over $\K$ by stable rules.
\end{theorem}

\begin{theorem} \label{4: Thm K4m1 stable logic and stable formula}
    Let $m \geq 1$. Every $\Kff{m+1}{1}$-stable logic is axiomatizable by stable formulas over $\Kff{m+1}{1}$. 
\end{theorem}

Note that the converse of \Cref{4: Thm K4m1 stable logic and stable formula} does not hold. A logic axiomatized by a stable formula over $\Kf$ that is not $\Kf$-stable can be found in \cite[Example 4.11]{stablemodallogic}.

If we take $D$ to be the entire set $A$ in a stable canonical rule or formula, then its semantic characterization becomes similar to that of \emph{Jankov formulas} \cite{jankov1963,dejongh1968}. This motivates the following terminology. We call a stable canonical rule of the form $\rho(\A, A)$ a \emph{Jankov rule} and a stable canonical formula of the form $\gamma^m(\A, A)$ or $\epsilon(\A, A)$ a \emph{Jankov formula}. The notion of \emph{splittings} and \emph{union-splittings} comes from lattice theory, and they have been an important tool in the study of lattices of logics (see \cite[Section 10.7]{czModalLogic1997} for a historical overview). It turns out that Jankov rules and Jankov formulas axiomatize exactly union-splittings in corresponding lattices. 

Recall that in a complete lattice $X$, a \emph{splitting pair} of $X$ is a pair $(x, y)$ of elements of $X$ such that $x \not\leq y$ and for any $z \in X$, either $x \leq z$ or $z \leq y$. If $(x, y)$ is a splitting pair of $X$, we say that $y$ \emph{splits} $X$ and $x$ is a \emph{splitting} in $X$. We call $z$ a \emph{union-splitting} in $X$ if $z$ is the join of a set of splittings in $X$. Since the lattices $\NExt{\S}$ for a rule system $\S$ and $\NExt{L}$ for a logic $L$ are complete, these notions apply. The following theorems are from \cite{stablecanonicalrules}, \cite{blok1978degree}, and \cite{takahashi2025stablecanonicalrulesformulas} in order.

\begin{theorem} \label{4: Thm splitting rule system and Jankov rule}
    A rule system $\S$ is a splitting (resp. a union-splitting) in $\NExt{\S_\K}$ iff $\S$ is axiomatizable over $\S_\K$ by a Jankov rule (resp. Jankov rules). 
\end{theorem}

\begin{theorem} \label{4: Thm K splitting and Jankov formula}
    A logic $L$ is a splitting (resp. a union-splitting) in $\NExt{\K}$ iff $L$ is axiomatizable over $\K$ by a Jankov formula (resp. Jankov formulas) of the form $\epsilon(\A, A)$. 
\end{theorem}

\begin{theorem} \label{4: Thm K4m1 splitting and Jankov formula}
    Let $m \geq 1$. A logic $L \in \NExt{\Kff{m+1}{1}}$ is a splitting (resp. a union-splitting) in $\NExt{\Kff{m+1}{1}}$ iff $L$ is axiomatizable over $\Kff{m+1}{1}$ by a Jankov formula (resp. Jankov formulas) of the form $\gamma^m(\A, A)$. 
\end{theorem}

\section{The Subdivision Construction} \label{Sec 3}

In this section, we focus on the combinatorics of modal spaces and introduce the \emph{Subdivision Construction}. Our goal is to prove the main lemma, \Cref{4: Lem main}. The usefulness and applications of the Subdivision Construction and the main lemma will become clear in the subsequent sections.

\subsection{The rank function}

We begin by introducing a measure that will be used in our inductive construction. For a modal space $\X$ and $x \in X$, we define $\r{x}$, the \emph{rank} of $x$ (in $\X$), to be the length of the longest path starting at $x$: formally, $\r{x} = \sup \{n \geq 1: \X, x \not\models \Boxx{n-1} \bot\}$. For example, $\r{x} = 1$ iff $x$ is a dead end, and $\r{x} = \omega$ iff there is an arbitrarily long path starting at $x$. It is easy to see that the rank $<\kappa$ ($\kappa \leq \omega$) part of a modal space is an upset, and the rank is non-decreasing under stable maps.

\begin{lemma}
    Let $\X = (X, R)$ be a modal space and $\kappa \leq \omega$. Then the rank $< \kappa$ part of $\X$ is an upset of $\X$.
\end{lemma}

\begin{proof}
    Let $x, y \in \X$ such that $xRy$ and $\r{x} < \kappa$. For any finite path starting at $y$, adding $x$ at the beginning results in a longer path starting at $x$, so $\r{y} < \r{x}  < \kappa$. Thus, the rank $< \kappa$ part of $\X$ is an upset of $\X$.
\end{proof}

\begin{lemma} \label{4: Lem stable map rank non-decrease}
    Let $\X = (X, R)$ and $\Y = (Y, Q)$ be modal spaces and $f: \X \to \Y$ be a stable map. Then, $\r{x} \leq \r{f(x)}$ for any $x \in \X$. 
\end{lemma}

\begin{proof}
    Let $f: \X \to \Y$ be a stable map and $x \in \X$. Suppose that $\r{x} = n$. Then there is a path $x=x_0 R \cdots R x_{n-1}$ in $\X$. Applying the stable map $f$, we obtain the path $f(x)=f(x_0) Q \cdots Q f(x_{n-1})$ in $\Y$. Thus, there is a path of length $n$ starting at $f(x)$, and we conclude $\r{f(x)} \geq n$. Suppose that $\r{x} = \omega$. Then there is a path of length $n$ starting at $x$ for any $n \in \omega$. Applying the same argument as above, we obtain a path of length $n$ starting at $f(x)$ for any $n \in \omega$, and we conclude $\r{f(x)} = \omega$.
\end{proof}

A \emph{cycle} in a modal space $(X, R)$ is a non-empty set $\{x_0, \dots, x_n\}$ of points such that $x_i R x_{i+1}$ for $0 \leq i \leq n$ where we regard $x_{n+1}$ as $x_0$. A modal space $\X$ is \emph{cycle-free} if there is no cycle in $\X$. Note that a cycle-free modal space contains no reflexive point. For finite modal spaces, being cycle-free admits a useful characterization via the rank function.

\begin{proposition} \label{4: Prop cyclefree char}
    Let $\X = (X, R)$ be a finite modal space. The following are equivalent.
    \begin{enumerate}
        \item $\X$ is cycle-free,
        \item $\r{x} < \omega$ for all $x \in \X$,
        \item there is some $N \in \omega$ such that $\r{x} \leq N$ for all $x \in \X$.
    \end{enumerate}
\end{proposition}

\begin{proof}
    If $\X$ has a cycle, then any point in the cycle has rank $\omega$, so $(2) \Rightarrow (1)$ follows. If $\r{x} = \omega$ for some $x \in \X$, then there is no $N \in \omega$ such that $\r{x} \leq N$, so $(3) \Rightarrow (2)$ follows. Finally, suppose that there is no $N \in \omega$ such that $\r{x} \leq N$ for all $x \in \X$. So, there is some $x \in \X$ such that $\r{x} > |X|$. Then there is a path of length $>|X|$ (starting at $x$) in $\X$. Such a path must contain two identical points, which implies that $\X$ has a cycle. This shows $(1) \Rightarrow (3)$.
\end{proof}

\subsection{The Subdivision Construction: intuition}

The idea of the Subdivision Construction is as follows. Suppose that we want to show the fmp of a logic $L$. Let $\phi \notin L$ and $\X$ be an $L$-space that refutes $\phi$. According to the theory of stable canonical rules, we may assume $\phi$ to be a stable canonical rule $\rho(\F, \D)$, and thus there is a stable map $f: \X \surj_\D \F$. Using the information of $f$, we construct a finite modal space $\F'$ out of $\F$ such that $\F'$ retains more of the structure of $\X$ and $\F' \surj_\D \F$. This is done by dividing the preimage under $f$ of points in $\F$ in an appropriate way, which is why we name the construction \emph{Subdivision}. The output $\F'$ will be a finite $L$-space refuting $\phi$. Before introducing the construction in its full details, we present a simple example that illustrates how the construction works and the intuition behind it. 

\begin{example} \label{Ex subdivision}

Let $\F_1$ be the following non-transitive modal space.
\begin{figure}[H]
    \centering
    \begin{tikzpicture}[scale=1]
        \node[label=right:{\(a\)}] (1) at (0,0) {\(\bullet\)};
        \node[label=right:{\(b\)}] (2) at (0,1) {\(\bullet\)};
        \node[label=right:{\(c\)}] (3) at (0,2) {\(\bullet\)};
        \draw[->] (1) -- (2);
        \draw[->] (2) -- (3);
    \end{tikzpicture}
\end{figure}

Let $\X = (X, R)$ be a modal space and $f_1: \X \surj_\emp \F_1$ be a stable map. All points in $f_1^{-1}(c)$ are dead ends, whereas there are two types of points in $f_1^{-1}(b)$: those that see some point in $f_1^{-1}(c)$ and those that are dead ends. However, this distinction is lost when applying $f_1$. We recover this information by letting $f_1$ factor through the following modal space $\F_2$.

\begin{figure}[H]
    \centering
    \begin{tikzpicture}[scale=1]
        \node[label=right:{\(a\)}] (1) at (0,-0.5) {\(\bullet\)};
        \node[label=right:{\(b_1\)}] (2) at (0,1) {\(\bullet\)};
        \node[label=right:{\(b_2\)}] (21) at (1,1) {\(\bullet\)};
        \node[label=right:{\(c\)}] (3) at (0,2) {\(\bullet\)};
        \draw[->] (1) -- (2);
        \draw[->] (1) -- (21);
        \draw[->] (2) -- (3);
    \end{tikzpicture}
\end{figure}

Define the stable map $f_2: \X \surj_\emp \F_2$ by $f_2(x) = b_1$ if $x \in f_1^{-1}(b)$ and $x$ sees some point in $f_1^{-1}(c)$, $f_2(x) = b_2$ if $x \in f_1^{-1}(b)$ and $x$ is a dead end, and $f_2(x) = f_1(x)$ otherwise. One can verify that $f_2$ is continuous, surjective (assuming that both types of points exist in $f_1^{-1}(b)$), and stable. Also, there is a canonical stable map $g_2: \F_2 \surj_\emp \F_1$, identifying $b_1$ and $b_2$ as $b$. Indeed, we have $f_1 = g_2 \circ f_2$. Moreover, all points in $f_2^{-1}(b_1)$ see some points in $f_2^{-1}(c)$ and all points in $f_2^{-1}(b_2)$ are dead ends. In this sense, we recover the information that was lost when applying $f_1$ directly.

Next, we proceed with $f_2^{-1}(a)$ similarly. Formally, there are eight types of points in $f_2^{-1}(a)$, depending on whether they see some point in $f_2^{-1}(c)$, $f_2^{-1}(b_1)$, and $f_2^{-1}(b_2)$. But since $f_2$ is stable, a point in $f_2^{-1}(a)$ cannot see any point in $f_2^{-1}(c)$ as $a$ does not see $c$. So, only four types are realizable. Assuming that all four types of points exist in $f_2^{-1}(a)$, we may consider the following modal space $\F_3$ and construct a similar factorization of $f_2$ as $g_3 \circ f_3$.

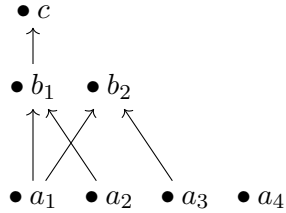
\begin{figure}[H]
    \centering
    \begin{tikzpicture}[scale=1]
        \node[label=right:{\(a_1\)}] (11) at (0,-0.5) {\(\bullet\)};
        \node[label=right:{\(a_2\)}] (12) at (1,-0.5) {\(\bullet\)};
        \node[label=right:{\(a_3\)}] (13) at (2,-0.5) {\(\bullet\)};
        \node[label=right:{\(a_4\)}] (14) at (3,-0.5) {\(\bullet\)};
        \node[label=right:{\(b_1\)}] (21) at (0,1) {\(\bullet\)};
        \node[label=right:{\(b_2\)}] (22) at (1,1) {\(\bullet\)};
        \node[label=right:{\(c\)}] (3) at (0,2) {\(\bullet\)};
        \draw[->] (11) -- (21);
        \draw[->] (11) -- (22);
        \draw[->] (12) -- (21);
        \draw[->] (13) -- (22);
        \draw[->] (21) -- (3);
    \end{tikzpicture}
\end{figure}

The finite modal space $\F_3$ and the maps $f_3: \X \surj_\emp \F_3$ and $g_3: \F_3 \surj_\emp \F_1$ are the results of the Subdivision Construction applied to the map $f_1: \X \surj_\emp \F_1$. One can verify from the construction that $f_3$ is in fact a p-morphism. Thus, $\F_3$ respects the structure of $\X$ more than $\F_1$ does.
\end{example}

This example raises two possible concerns about the Subdivision Construction. First, not all types may be realized in a preimage. For example, there might be no dead end in $f_2^{-1}(a)$, so $f_3$ might not be surjective. This issue can be overcome by adding points only when points of the corresponding type exist in the preimage. One may also wonder what happens if we take the closed domain $\D$ into account. As we will see in the following, the CDC works perfectly with this construction. 

Moreover, it should be clear from this example that the Subdivision Construction applies practically only to the cycle-free part, in particular, the irreflexive part, of a modal space. If $a$ is reflexive, then points in $f_2^{-1}(a)$ can see each other, and there is no appropriate way to define the relation in $\F_3$ on $\{a_1, a_2, a_3, a_4\}$. 

\subsection{The Subdivision Construction: formal definition}

\begin{definition}
    Let $\X = (X, R)$ and $\Y = (Y, Q)$ be modal spaces and $f: \X \to \Y$ be a stable map. 
    \begin{enumerate}
        \item For a point $x \in \X$, we call $f$ \emph{a p-morphism for $x$} if, for any $y' \in Y$ such that $f(x)Qy'$, there exists some $x' \in X$ such that $xRx'$ and $f(x') = y'$.
        \item For a subset $U \subseteq \X$, we call $f$ \emph{a p-morphism for $U$} if $f$ is a p-morphism for $x$ for all $x \in U$.
    \end{enumerate}
\end{definition}

A stable map $f: \X \to \Y$ is a p-morphism for $X$ iff it is a p-morphism in the usual sense.

Now we are ready to formally introduce the \emph{Subdivision Construction}. The output of the construction is summarized in the following lemma, whose proof will be given along the way. 

\begin{lemma}[Subdivision Lemma] \label{4: Lem main}
    Let $\X$ be a modal space, $\F$ be a finite modal space, and $f: \X \surj_\emp \F$ be a stable map. Then, the Subdivision Construction applied to $f$ produces
    \begin{enumerate}
        \item a finite modal space $\F'$,
        \item a stable map $f': \X \surj_\emp \F'$ such that $f'$ is a p-morphism for $x \in \X$ with $\r{f'(x)} < \omega$,
        \item a stable map $g: \F' \surj_\emp \F$ such that $g$ is an isomorphism between the rank $\omega$ part of $\F'$ and that of $\F$,
    \end{enumerate}
    such that the following diagram commutes.

    \begin{center}
        \adjustbox{scale=1,center}{
\begin{tikzcd}
\X \arrow[rr, "f"] \arrow[rd, "f'"'] &                      & \F \\
                                     & \F' \arrow[ru, "g"'] &   
\end{tikzcd}
        }
    \end{center}
    Moreover, if $f$ satisfies CDC for $\D$ for some $\D \subseteq \pow{F}$, then so does $g$.
\end{lemma}


Let $\X = (X, R)$ be a modal space, $\F = (F, Q)$ be a finite modal space, and $f: \X \surj_\emp \F$ be a stable map. We divide the preimage of points in $\F$ inductively on rank by assigning to each point $x \in \X$ a \emph{signature} $s(x)$. The intuitive reading is that for two points $x, y \in f^{-1}(w)$ for some $w \in \F$, they are in the same subdivision of $f^{-1}(w)$ iff $s(x) = s(y)$. These signatures will be the ``name'' of the represented subdivision. The new finite space $\F'$ will have these signatures $s(x)$ for $x \in f^{-1}(w)$ as its points instead of the original $w$. 

For $x \in \X$ with $\r{f(x)} = 1$, let $s(x) = \<f(x), \emp\>$. This aligns with the fact that the dead end $c$ remained unchanged in \Cref{Ex subdivision}. Assuming that $s(x)$ is defined for $x \in \X$ with $\r{f(x)} < n$, for $x \in \X$ with $\r{f(x)} = n$, let $s(x) = \<f(x), \{s(x'): xRx'\}\>$. This is well-defined because for any $y, y' \in \X$ with $\r{f(y)} = n$ and $yRy'$, we have $f(y)Qf(y')$ since $f$ is stable, so $\r{f(y')} < \r{f(y)} = n$ and $s(y')$ is defined by assumption. Recall that in \Cref{Ex subdivision} we classified points in the preimage of $a$ according to whether they see some points in the preimage of $c$, $b_1$, and $b_2$. This data is recorded in the second coordinate of the signiture of a point. We have inductively defined the subdivision of the preimage of points in $\F$ with a finite rank. For $x \in \X$ with $\r{f(x)} = \omega$, let $s(x) = \<f(x), \emp\>$, which means that the rank $\omega$ part of $\F$ remains unchanged. 

Let $\F' = (F', Q')$, where $\F' = \{s(x): x \in \X\}$, and $Q'$ is defined by $s(x)Q's(y)$ iff 
\begin{enumerate}
    \item[(i)] $\r{f(x)} < \omega$, $\r{f(y)} < \omega$, and $s(y) \in \{s(x'): xRx'\}$; or
    \item[(ii)] $\r{f(x)} = \omega$ and $f(x)Qf(y)$. 
\end{enumerate}
The definition of $Q'$ only refers to $s(x)$ and $s(y)$, as well as their components, but not $x$ and $y$ themselves, so it is well-defined. Define $f': \X \to \F'$ by $f'(x) = s(x)$ and let $g: \F' \to \F$ be the projection to the first coordinate, that is, $g(\<w, S\>) = w$. We show that $\F'$, $f'$, and $g$ as defined satisfy the conditions in \Cref{4: Lem main} by a series of claims. 

The following observation will be useful for proving the claims.

\begin{lemma} \label{Lem for claims}
    Let $\<w, S\> \in \F'$. If $\r{w} = \omega$, then $S = \emp$ and $\r{\<w, S\>} = \omega$; if $\r{w} < \omega$, then $\r{\<w, S\>} \leq \r{w}$, in particular, $\r{\<w, S\>} < \omega$.
\end{lemma}

\begin{proof}
    Let $\<w, S\> \in \F'$ such that $\r{w} = \omega$. It follows directly from the construction that $S = \emp$. Also, by $\r{w} = \omega$, there is an infinite sequence $w = w_0 Q w_1 Q \dots$ in $\F$, where each $w_i$ has rank $\omega$. By the definition of $Q'$, we obtain an infinite sequence $\<w, S\> = \<w_0, \emp\> Q' \<w_1, \emp\> Q' \dots$ in $\F'$. Thus, $\r{\<w, S\>} = \omega$. 

    Next, we show that $\r{\<w, S\>} \leq \r{w}$ if $\r{w} < \omega$ by induction on $\r{w}$. If $\r{w} = 1$, then $S = \emp$ by the construction, so $\<w, S\>$ has no successor by the definition of $Q'$, hence $\r{\<w, S\>} = 1 = \r{w}$. Assume that $\r{\<w', S'\>} \leq \r{w'}$ for any $\<w', S'\> \in \F'$ such that $\r{w'} < n$. Let $\<w, S\> \in \F'$ such that $\r{w} = n$. If $\<w, S\> Q' \<w', S'\>$, then by the definition of $Q'$, we have $\<w', S'\> \in S$. By the construction, there is some $x \in \X$ such that $s(x) = \<w, S\>$ and $\{s(x'): xRx'\} = S$. So, there is some $x' \in \X$ such that $xRx'$ and $s(x') = \<w', S'\>$. This implies $f(x) = w$ and $f(x') = w'$. Since $f$ is stable, we have $wQw'$ by $xRx'$, so $\r{w'} < \r{w}$. By the induction hypothesis, $\r{\<w', S'\>} \leq \r{w'} < \r{w}$. Thus, $\r{\<w, S\>} \leq \r{w}$. This completes the induction and the proof of the claim.
\end{proof}

\begin{claim}
    $\F'$ is a finite modal space.
\end{claim}

\begin{proof}
    It remains only to check that $\F'$ is finite. Observe that $|\{s(x): \r{f(x)} = 1\}| \leq |F|$; if $|\{s(x): \r{f(x)} < n\}| = m$, then $|\{s(x): \r{f(x)} = n\}| \leq |F| \times 2^m$, as for each $x \in \X$ with $\r{f(x)} = n$, we have $\{s(y): xRy\} \subseteq \{s(y): \r{f(y)} < n\}$. This shows that for any $n$, $|\{s(x): \r{f(x)} \leq n\}|$ is finite. It follows from a similar argument as in the proof of \Cref{4: Prop cyclefree char} that, since $\F$ is finite, the rank $< \omega$ part of $\F$ has a uniform upper bound on rank, say, $N$. Since clearly $|\{s(x): \r{f(x)} = \omega\}| \leq |F|$ by the construction, we have $|F'| \leq |\{s(x): \r{f(x)} \leq N\}| + |F|$, which is finite.
\end{proof}

\begin{claim}
    $g \circ f' = f$.
\end{claim}

\begin{proof}
    For any $x \in \X$, we have $g(f'(x)) = g(\<f(x), S\>) = f(x)$ for some $S \subseteq F'$. Thus, $g \circ f' = f$.
\end{proof}

\begin{claim}
    $f'$ is a surjective stable map.
\end{claim}

\begin{proof}
    The surjectivity of $f'$ follows immediately from the construction of $\F'$ and the definition of $f'$. To show that $f'$ is continuous, it suffices to show that the preimage of each point in $\F'$ is clopen in $\X$. Let $\<w, S\> \in \F'$ so that $w \in \F$ and $S \subseteq F'$. If $\r{w} = 1$, then $f'^{-1}(\<w, S\>) = f^{-1}(w)$, which is clopen since $f$ is continuous. Assume that $f'^{-1}(\<w', S'\>)$ is clopen for any $w'$ with $\r{w'} < n$. If $\r{w} = n$, then 
    \begin{align*}
        &f'(x) = \<w, S\>    \text{ iff } \\ 
        &\begin{cases}
            f(x) = w \text{; and} &\text{}\\
            \forall \<w', S'\> \in \F' (\<w', S'\> \in S \to \exists x' \in \X (xRx' \land f'(x') = \<w', S'\>)) \text{; and} &\text{}\\
            \forall \<w', S'\> \in \F' (\<w', S'\> \notin S \land \r{w'} < n \to \lnot \exists x' \in \X (xRx' \land f'(x') = \<w', S'\>)). &\text{}\\
        \end{cases}
    \end{align*}
    In the last clause, we may restrict to $w'$ with $\r{w'} < n$ because otherwise the conclusion always holds: if $f(x) = w$ and $\r{w'} \geq n$, then $\exists x' \in \X (xRx' \land f'(x') = \<w', S'\>)$ would imply $f(x)Qf(x')$ and $f(x') = w'$, so $\r{w} = \r{f(x)} > \r{f(x')} = \r{w'} \geq n$, which contradicts $\r{w} = n$. The above equivalence implies 
    \begin{align*}
        f'^{-1}(\<w, S\>) = &f^{-1}(w) \cap \\
        &\bigcap \{R^{-1}(f'^{-1}(\<w', S'\>)): \<w', S'\> \in S\} \cap \\
        &\bigcap \{X \setminus R^{-1}(f'^{-1}(\<w', S'\>)): \<w', S'\> \notin S, \r{w'} < n\},
    \end{align*}
which, by our assumption, is a finite intersection of clopen sets and thus clopen. Inductively, this shows that if $\<w, S\> \in \F'$ and $\r{w} < \omega$, then $f'^{-1}(\<w, S\>)$ is clopen. Finally, if $\<w, S\> \in \F'$ and $\r{w} = \omega$, then $S = \emp$ by the construction, and we have $f'^{-1}(\<w, S\>) = f^{-1}(w)$, which is clopen since $f$ is continuous. So, $f'$ is continuous. 

To show the stability of $f'$, let $x, y \in \X$ such that $xRy$. Then $f(x)Qf(y)$ since $f$ is stable. If $\r{f(x)} < \omega$, then $\r{f(y)} < \omega$ as well, so $f'(x)Q'f'(y)$ since $xRy$ implies $s(y) \in \{s(x'): xRx'\}$. If $\r{f(x)} = \omega$, then $f'(x)Q'f'(y)$ follows directly from $f(x)Qf(y)$. Thus, $f'$ is stable.
\end{proof}

\begin{claim}
    $f'$ is a p-morphism for $x \in \X$ with $\r{f'(x)} < \omega$. 
\end{claim}

\begin{proof}
    Let $x \in \X$ with $\r{f'(x)} < \omega$. Then, by \Cref{Lem for claims}, $\r{f(x)} < \omega$ as well. Suppose that $f'(x)$ has a $Q'$-successor. Then it must be $f'(y)$ for some $y \in \X$ since $f'$ is surjective. Since $\r{f(x)} < \omega$, it follows from the definition of $Q'$ that $\r{f(y)} < \omega$ and $s(y) \in \{s(x'): xRx'\}$. This means that there is some $x' \in \X$ such that $xRx'$ and $f'(x') = f'(y)$. Thus, $f'$ is a p-morphism for $x$. 
\end{proof}

\begin{claim}
    $g$ is a surjective stable map.
\end{claim}

\begin{proof}
    The surjectivity of $g$ follows immediately from the surjectivity of $f$. Since $\F'$ is finite, it has the discrete topology, so $g$ is continuous. To show the stability of $g$, let $s, s' \in \F'$ such that $sQ's'$. Since $f'$ is surjective, $s = f'(x)$ and $s' = f'(x')$ for some $x, x' \in \X$. If $\r{s} < \omega$, then, since $f'$ is a p-morphism for $x \in \X$ with $\r{f'(x)} < \omega$, we may choose $x'$ to satisfy $xRx'$. Then, since $f$ is stable, we have $f(x)Qf(x')$, namely $g(s)Qg(s')$. If $\r{s} = \omega$, then $\r{f(x)} = \omega$ as well by \Cref{Lem for claims}, so $f(x)Qf(x')$ by the definition of $Q'$, namely $g(s)Qg(s')$. Thus, $g$ is stable.
\end{proof}

\begin{claim}
    $g$ is an isomorphism between the rank $\omega$ part of $\F'$ and that of $\F$.
\end{claim}

\begin{proof}
    It follows from \Cref{Lem for claims} that the rank $\omega$ part of $\F'$ has the underlying set $\{\<w, \emp\>: w \in \F, \r{w} = \omega\}$, and the statement follows immediately from $g(\<w, \emp\>) = w$, the definition of $g$.
\end{proof}

\begin{claim}
    Let $D \subseteq F$. If $f$ satisfies CDC for $D$, then so does $g$. Consequently, for $\D \subseteq \pow{F}$, if $f$ satisfies CDC for $\D$, then so does $g$. 
\end{claim}

\begin{proof}
    Let $D \subseteq F$ and suppose that $f$ satisfies CDC for $D$. To show that $g$ satisfies CDC for $D$, let $s \in \F'$ such that $g(s) Q w'$ for some $w' \in D$. Since $f'$ is surjective, there is some $x \in \X$ such that $f'(x) = s$, and so $f(x) = g(f'(x)) = g(s)$. Since $f$ satisfies CDC for $D$, there is some $x' \in \X$ such that $xRx'$ and $f(x') \in D$. Then $sQ' f'(x')$ since $f'$ is stable, and $g(f'(x')) = f(x') \in D$. Thus, $g$ satisfies CDC for $D$. The second statement follows immediately from the first one.
\end{proof}

The preceding claims complete the proof of \Cref{4: Lem main}.

\begin{remark}
    The Subdivision Construction can also be described in a step-by-step way inductively on the rank of points in $\F$: at the $n$-th step, we replace the points in $\F$ that have rank $n$ with new points corresponding to the subdivisions of their preimages. This presentation reflects the idea shown in \Cref{Ex subdivision} more directly, but it requires much more complicated notation, and thus we chose the current presentation of the construction as above. A detailed step-by-step description of the Subdivision Construction, as an alternative proof of \Cref{4: Lem main}, can be found in \cite[Chapter 4]{TakahashiThesis}. 
\end{remark}

\begin{remark}
    The idea of the Subdivision Construction is inspired by the direct proof of \Cref{4: Cor K-union-splittings have fmp} given in \cite[Theorem 10.54]{czModalLogic1997}, which was originally proved by Blok \cite{blok1978degree}. While the proof of \cite[Theorem 10.54]{czModalLogic1997} applies a certain filtration to the higher rank part of $\X$ and leaves the lower rank part of $\X$ unchanged so that the filtration would not destroy the structure of $\X$ too much, the Subdivision Construction can be seen as a post-processing of a standard filtration $\F$ of $\X$, leaving the higher rank part of $\F$ unchanged and modifying the lower rank part of $\F$ to recover the structure of $\X$ destroyed by the applied standard filtration.
\end{remark}

\section{Finite model property via the Subdivision Construction} \label{Sec 4}

In this section, we demonstrate how to use the Subdivision Construction (\Cref{4: Lem main}) to prove the fmp for logics and rule systems. In general, the Subdivision Construction does not preserve properties of modal spaces due to its combinatorial nature. When it does, it will produce a finite countermodel that is useful for showing the fmp. 

\begin{definition} \label{4: Def Subdivision preserve} \leavevmode
    \begin{enumerate}
        \item Let $\S$ and $\S'$ be rule systems such that $\S \subseteq \S'$. We say that the Subdivision Construction \emph{preserves $\S'$ over $\S$} if, whenever $\X \models \S'$ and $\F \models \S$ where $\X$ and $\F$ are as in the assumption of \Cref{4: Lem main}, the modal space $\F'$ obtained in \Cref{4: Lem main} validates $\S'$. 
        \item Let $L$ and $L'$ be logics such that $L \subseteq L'$. We say that the Subdivision Construction \emph{preserves $L'$ over $L$} if, whenever $\X$ is topo-rooted (i.e., the dual algebra is s.i.), $\X \models L'$, and $\F \models L$, where $\X$ and $\F$ are as in the assumption of \Cref{4: Lem main}, the modal space $\F'$ obtained in \Cref{4: Lem main} validates $L'$. 
    \end{enumerate}
\end{definition}

Note that in the above definition for rule systems, we cannot restrict to topo-rooted spaces (s.i.~algebras) because, contrary to varieties, universal classes may not be generated by their s.i.~members.

We will prove the fmp for logics and rule systems of the following form, and observe that they are union-splittings in the corresponding lattices of rule systems or logics. 
\begin{enumerate}
    \item A stable rule system + stable canonical rules of finite modal algebras of finite height,
    \item A strictly stable logic + stable canonical formulas of finite s.i.~modal algebras of finite height,
    \item A $\Kff{m+1}{1}$-stable logic + stable canonical formulas of finite s.i.~$\Kff{m+1}{1}$-algebras of finite height.
\end{enumerate}

The fmp of these logics and rule systems is proved via the following general scheme. 

\begin{theorem} \label{4: Thm main scheme} \leavevmode
\begin{enumerate}
    \item Let $\S$ be a rule system that admits definable filtration, and $\S' \supseteq \S$ be a rule system such that the Subdivision Construction preserves $\S'$ over $\S$. Then, $\S'$ has the fmp. 
    \item Let $L$ be a logic that admits definable filtration, and $L' \supseteq L$ be a logic such that the Subdivision Construction preserves $L'$ over $L$. Then, $L'$ has the fmp. 
\end{enumerate}
\end{theorem}

\begin{proof}
    We only prove the statement for logics. The proof for rule systems is similar.
    
    Let $L$ be a logic that admits definable filtration, and $L' \supseteq L$ be a logic such that the Subdivision Construction preserves $L'$ over $L$. Let $\phi \notin L'$. It follows from the fact that every variety is generated by its s.i.~members (see, e.g., \cite[Theorem 3.44]{UniversalAlgebraFundamentals2011}) that there is a s.i.~modal algebra $\B$ such that $\B \models L'$ and $\B \not\models \phi$. Since $L$ admits definable filtration, by \Cref{3: Thm scr complete}, $\phi$ is equivalent to a set of stable canonical rules $\{\rho(\A_k, D_k): 1 \leq k \leq n\}$ where each $\A_k$ is a finite $L$-algebra and $D_k \subseteq A_k$. So, $\B \not\models \rho(\A_k, D_k)$ for some $1 \leq k \leq n$, namely $\A_k \inj_{D_k} \B$. 

    Let $\X = (X, R)$ and $\F = (F, Q)$ be the dual space of $\B$ and $\A_k$ respectively, and $\D = \sigma[D_k]$. Then, $\X$ is topo-rooted, $\X \models L'$, $\F \models L$, and there is a stable map $f: \X \surj_{\D} \F$. Applying the Subdivision Construction (\Cref{4: Lem main}), we obtain a finite modal space $\F'$ and a stable map $g: \F' \surj_{\D} \F$. Let $\A'$ be the dual algebra of $\F'$. Since the Subdivision Construction preserves $L'$ over $L$, we have $\F' \models L'$, so $\A'$ is a finite $L'$-algebra. Moreover, since $\F' \surj_{\D} \F$, dually we have $\A_k \inj_{D_k} \A'$, so $\A' \not\models \rho(\A_k, D_k)$, hence $\A' \not\models \phi$. Thus, $\A'$ is a finite $L'$-algebra refuting $\phi$, and therefore we conclude that $L'$ has the fmp.
\end{proof}

Before proving the fmp results, we collect some facts on modal algebras of finite height and cycle-free modal spaces for later use. Let $\FH$ be the class of modal algebras of finite height, that is, 
\begin{align*}
    \FH &= \{\A: \exists n \in \omega  (\A \models \Box^n\bot)\} \\
        &= \{\A: \exists n \in \omega  (\Box^n 0 = 1)\}.
\end{align*}
Let $\FH\fin$ be the class of finite members of $\FH$.

\begin{proposition} \label{4: Prop cf dual to fh}
    Let $\A$ be a finite modal algebra and $\X$ be its dual space. Then, $\A$ is of finite height iff $\X$ is cycle-free. Therefore, $\FH\fin$ is the class of modal algebras dual to finite cycle-free modal spaces.
\end{proposition}

\begin{proof}
    Let $\A$ be a finite modal algebra and $\X$ be its dual space. Then, 
    \begin{align*}
        \A \text{ is of finite height} 
        &\iff \text{there is some $n \in \omega$ such that $\A \models \Box^n\bot$} \\
        &\iff \text{there is some $n \in \omega$ such that $\X \models \Box^n\bot$} \\
        &\iff \text{there is some $n \in \omega$ such that $\r{x} \leq n$ for all $x \in \X$} \\
        &\iff \X \text{ is cycle-free (by \Cref{4: Prop cyclefree char}).}
    \end{align*}
    Thus, a finite modal algebra is of finite height iff its dual modal space is cycle-free.
\end{proof}

Stable subalgebras reflect the property of having finite height. 

\begin{lemma} \label{4: Lem stable subalgebra reflects fh}
    Let $\A$ be a modal algebra and $\B$ be a stable subalgebra of $\A$. If $\B$ is of finite height, then so is $\A$. Therefore, $\FH$ reflects stable subalgebras.
\end{lemma}

\begin{proof}
    Let $i: \B \to \A$ be a stable embedding. Suppose $\B$ is of finite height. Then there is some $n \in \omega$ such that $\Box^n 0_\B = 1_\B$. So, 
    \[\Box^n 0_\A = \Box^ni(0_\B) \geq i(\Box^n 0_\B) = i(1_\B) = 1_\A,\]
    hence $\Box^n 0_\A = 1_\A$. Thus, $\A$ is of finite height.
\end{proof}

Dually, the above lemma means that stable maps reflect the property of being cycle-free, which is also a consequence of \Cref{4: Lem stable map rank non-decrease}.

\subsection{Union-splittings of the lattice of extensions of stable rule systems}

Now we proceed to apply \Cref{4: Thm main scheme} to show the fmp for logics and rule systems. We start with the fmp for a class of rule systems. This case is relatively more straightforward because stable canonical rules have a simpler semantic characterization than stable canonical formulas.

\begin{lemma} \label{4: Lem stable canonical rule preserved}
    Let $\S$ be a stable rule system. If $\{\rho(\A_i, D_i): i \in I\}$ is a set of stable canonical rules where each $\A_i$ is a finite modal algebra of finite height and $D_i \subseteq A_i$, then the rule system 
    \[\S + \{\rho(\A_i, D_i): i \in I\}\]
    is preserved by the Subdivision Construction over $\S_\K$.
\end{lemma}

\begin{proof}
    Let $\S' = \S + \{\rho(\A_i, D_i): i \in I\}$. Let $\X$ be an $\S'$-space, $\F$ be a finite modal space, $\D \subseteq \pow{F}$, and $f: \X \surj_\D \F$ be a stable map. Applying the Subdivision Construction (\Cref{4: Lem main}), we obtain a finite modal space $\F'$ and a stable map $f': \X \surj_\emp \F'$ such that $f'$ is a p-morphism for $x \in \X$ with $\r{f'(x)} < \omega$. Since $\S$ is stable, by the dual of \Cref{3: Def stable}, we have $\F' \models \S$. 

    Suppose for a contradiction that $\F' \not\models \rho(\A_i, D_i)$ for some $i \in I$. Let $\F_i$ be the dual space of $\A_i$ and $\D_i = \sigma[D_i]$. Then, $\F' \surj_{\D_i} \F_i$ by the dual of \Cref{3: Thm sc rule char}. Since $\A_i$ is of finite height, $\F_i$ is cycle-free by \Cref{4: Prop cf dual to fh}, and so is $\F'$ by the dual of \Cref{4: Lem stable subalgebra reflects fh}. It follows from \Cref{4: Prop cyclefree char} that every point $x \in \X$ satisfies $\r{f'(x)} < \omega$, so $f': \X \surj \F'$ is a p-morphism. Composing with a stable map $\F' \surj_{\D_i} \F_i$, we see that $\X \surj_{\D_i} \F_i$, namely, $\X \not\models \rho(\A_i, D_i)$, which contradicts $\X \models \S'$. Thus, $\F' \models \rho(\A_i, D_i)$ for all $i \in I$, and therefore $\F' \models \S'$. 
\end{proof}

\begin{theorem} \label{4: Thm fmp S + sp rule}
    Let $\S$ be a stable rule system. If $\{\rho(\A_i, D_i): i \in I\}$ is a set of stable canonical rules where each $\A_i$ is a finite modal algebra of finite height and $D_i \subseteq A_i$, then the rule system 
    \[\S + \{\rho(\A_i, D_i): i \in I\}\]
    has the fmp.
\end{theorem}

\begin{proof}
    This follows from \Cref{4: Thm main scheme} with \Cref{4: Lem stable canonical rule preserved} and the fact that $\S_\K$ admits (definable) filtration.
\end{proof}

\begin{remark}
    By \Cref{4: Thm stable rule system and logic and stable rule}, a rule system is stable iff it is axiomatizable over $\S_\K$ by stable rules. So, we can reformulate \Cref{4: Thm fmp S + sp rule} as follows: any rule system of the form 
    \[\S_\K + \{\rho(\A_j, \emp): j \in J\} + \{\rho(\A_i, D_i): i \in I\}\]
    where each $\A_j$ is a finite modal algebra and each $\A_i$ is a finite modal algebra of finite height, has the fmp.
\end{remark}

As a corollary, we obtain the fact that stable rule systems have the fmp \cite{stablecanonicalrules}.

\begin{corollary}
    Every stable rule system has the fmp.
\end{corollary}

\begin{proof}
    This follows from \Cref{4: Thm fmp S + sp rule} by taking $I = \emp$.
\end{proof}

A natural question is what these rule systems in \Cref{4: Thm fmp S + sp rule} are. It turns out that they form a lower part of the union-splittings of the lattice $\NExt{\S}$ and have a neat characterization. In particular, we observe that varying the closed domain condition and using finite modal algebras outside $\U(\S)$ do not increase axiomatization power, as long as the algebras are of finite height. Recall that $\FH$ is the class of modal algebras of finite height, and Jankov rules are stable canonical rules of the form $\rho(\A, A)$. The rule systems mentioned in \Cref{4: Thm fmp S + sp rule} are precisely those satisfying the condition (4) in the following theorem. Recall that for a class $\class{K}$ of modal algebras, $\class{K}\fin$ denotes the class of finite members of $\class{K}$.

\begin{theorem} \label{4: Thm characterization of splitting rule system}
    Let $\S$ be a stable rule system. For any rule system $\S' \in \NExt{\S}$, the following are equivalent:
    \begin{enumerate}
        \item $\U(\S) \subseteq \U(\S') \cup \FH$,
        \item $\U(\S)\fin \subseteq \U(\S') \cup \FH$,
        \item $\S'$ is axiomatized over $\S$ by Jankov rules of finite $\S$-algebras of finite height,
        \item $\S'$ is axiomatized over $\S$ by stable canonical rules of finite modal algebras of finite height,
        \item $\S' \subseteq \S + \{\rho(\A, A): \A \in \U(\S)\fin \cap \FH\}$.
    \end{enumerate}
    Moreover, any rule system $\S' \in \NExt{\S}$ satisfying (one of) the above conditions is a union-splitting in $\NExt{\S}$.
\end{theorem}

\begin{proof} 
    Let $\S$ be a stable rule system and $\S' \in \NExt{\S}$. We first show the equivalence of the first four items.
    
    $(1) \Rightarrow (2)$: This is clear. 

    $(2) \Rightarrow (3)$: Suppose that $\U(\S)\fin \subseteq \U(\S') \cup \FH$. Let 
    \[\S'' = \S + \{\rho(\A, A): \A \text{ is a finite $\S$-algebra such that } \A \not\models \S'\}.\] 
    Since $\U(\S)\fin \subseteq \U(\S') \cup \FH$, a finite $\S$-algebra refuting $\S'$ must be of finite height, so $\S''$ is axiomatized over $\S$ by Jankov rules of finite modal algebras of finite height, thus $\S''$ has the fmp by \Cref{4: Thm fmp S + sp rule}. It suffices to show that $\S' = \S''$.

    If $\S' \not\subseteq \S''$, then by the fmp of $\S''$, there is a finite $\S''$-algebra $\A$ such that $\A \not\models \S'$, which contradicts the definition of $\S''$ since $\A \not\models \rho(\A, A)$. So, $\S' \subseteq \S''$. Conversely, if $\S'' \not\subseteq \S'$, then there is an $\S'$-algebra $\B$ such that $\B \not\models \rho(\A, A)$ for some finite $\S$-algebra $\A$ such that $\A \not\models \S'$. So, $\A \inj \B$, and since $\U(\S')$ is closed under subalgebras, $\A \in \U(\S')$, which contradicts $\A \not\models \S'$. So, $\S'' \subseteq \S'$, and therefore $\S' = \S''$.

    $(3) \Rightarrow (4)$: This is clear.

    $(4) \Rightarrow (1)$: Let $\B \in \U(\S)$ such that $\B$ is not of finite height. For any finite modal algebra $\A$ of finite height, since being of finite height reflects stable subalgebras (\Cref{4: Lem stable subalgebra reflects fh}), $\A \not\inj_\emp \B$, so $\B \models \rho(\A, D)$ for any $D \subseteq A$. So, if $\S'$ is axiomatized over $\S$ by stable canonical rules of finite modal algebras of finite height, then $\B \models \S'$. Thus, $\U(\S) \subseteq \U(\S') \cup \FH$.

    Next, we show that $(3) \Rightarrow (5)$ and $(5) \Rightarrow (2)$, and thus completing the proof.

    $(3) \Rightarrow (5)$: This is clear.
    
    $(5) \Rightarrow (2)$: Let $\B \in \U(\S)\fin$ such that $\B$ is not of finite height. By a similar argument as in the case $(4) \Rightarrow (1)$, it follows that $\B \models \S + \{\rho(\A, A): \A \in \U(\S)\fin \cap \FH\}$, so $\B \models \S'$ by assumption (5). Thus, $\U(\S)\fin \subseteq \U(\S') \cup \FH$.

    Finally, as stable rule systems have the fmp, it is a straightforward generalization of \Cref{4: Thm stable rule system and logic and stable rule} (see also \cite[Theorem 7.10]{stablecanonicalrules}) that, $\S'$ is a union-splitting of $\NExt{\S}$ iff $\S'$ is axiomatizable over $\S$ by Jankov rules. So, the condition (3) implies being a union-splitting. Thus, any rule system $\S' \in \NExt{\S}$ satisfying (one of) the conditions (1)-(5) is a union-splitting in $\NExt{\S}$.
\end{proof}

The subclass of union-splittings of $\NExt{\S}$ in question is a lower part of the class of all union-splittings of $\NExt{\S}$ because of item (5) in \Cref{4: Thm characterization of splitting rule system}.

\begin{remark}
    For a finitely axiomatizable stable rule system $\S$, it also follows from \Cref{4: Thm characterization of splitting rule system} that this lower part of the union-splittings of $\NExt{\S}$ is decidable. More specifically, this means that there is an algorithm that, given finitely many rules $\rho_0, \dots, \rho_n$, decides whether the rule system $\S + \rho_0 + \dots + \rho_n$ falls in this class. Item (3) implies that the decision problem is recursively enumerable, and item (2) implies that the complement is recursively enumerable. A similar idea is used in \cite{takahashi2025decidabilityunionsplitting} to prove the decidability of being a union-splitting logic in $\NExt{\K}$.
\end{remark}

It is unknown if all union-splittings in $\NExt{\S}$ have the fmp for a stable rule system $\S$. This is expected since all finite $\S$-algebras split $\NExt{\S}$, which is similar to the case of $\NExt{\Kf}$ and contrary to the case of $\NExt{\K}$.

\subsection{Union-splittings of the lattice of extensions of strictly stable logics}

Next, we turn to logics. Our first application is intended to generalize Blok's fmp result that every union-splitting in $\NExt{\K}$ has the fmp. We start with a simple observation on $\Kff{m+1}{1}$-spaces. This is also helpful in the current case because the dual space of a modal algebra of finite height, in particular, a finite cycle-free modal space, is a $\Kff{m+1}{1}$-space for some $m \in \omega$.

\begin{lemma} \label{4: Lem K4m1 toporoot}
    Let $\X$ be a $\Kff{m+1}{1}$-space. A point $x \in \X$ is a topo-root iff it is a root.
\end{lemma}

\begin{proof}
    Since $\X$ is a $\Kff{m+1}{1}$-space, we have $R^{<\omega}[x] = R^{\leq m}[x]$, which is closed. So, the closure of $R^{<\omega}[x]$ coincides with itself, hence $x$ is a topo-root iff $x$ is a root.
\end{proof}

\begin{remark}
    It is clear from the proof that \Cref{4: Lem K4m1 toporoot} holds for a logic $L$ as long as there is some $m \in \omega$ such that any $L$-space satisfies $R^{<\omega}[x] = R^{\leq m}[x]$ for each point $x$ in the space. This condition holds for logics $\wKf = \K + \Dia \Dia p \to p \lor \Dia p$ and $\Kff{m}{n}$ where $n < m$ as well.
\end{remark}

Note that this does not imply that being topo-rooted and being rooted are equivalent for $\Kff{m+1}{1}$-spaces. To be topo-rooted, the set of topo-roots has to have a non-empty interior, which is stronger than just being non-empty.

Unlike the other two cases, we cannot deal with all stable logics because, in general, their characterization is not as strong as that of stable rule systems or $\Kf$-stable logics. If $L$ is a $\Kf$-stable logic, then the class $\V(L)\si$ is finitely $\Kf$-stable \cite{stablemodallogic}, that is, it is closed under finite stable $\Kf$-subalgebras. This fact generalizes to $\Kff{m+1}{1}$-stable logics (\Cref{4: Lem finitely K4m1 stable}). However, for a stable logic $L$, the class $\V(L)\si$ may not be closed under finite stable subalgebras. 

\begin{definition}
    A logic $L$ is \emph{strictly stable} if the class $\V(L)\si$ is finitely stable, that is, it is closed under finite stable subalgebras.
\end{definition}

Examples of strictly stable logics include $\K$, $\logic{KD}$, and $\logic{KT}$.

\begin{lemma} \label{4: Lem ss logic splitting preserved}
    Let $L$ be a strictly stable logic. If $\{\epsilon(\A_i, D_i): i \in I\}$ is a set of stable canonical formulas (so that each $\A_i$ is a finite s.i.~modal algebra of finite height and $D_i \subseteq A_i$), then the logic
    \[L + \{\epsilon(\A_i, D_i): i \in I\}\]
    is preserved by the Subdivision Construction over $\K$.
\end{lemma}

\begin{proof}
    Let $L' = L + \{\epsilon(\A_i, D_i): i \in I\}$. Let $\X = (X, R)$ be a topo-rooted $L'$-space, $\F$ be a finite modal space, $\D \subseteq \pow{F}$, and $f: \X \surj_\D \F$ be a stable map. Applying the Subdivision Construction (\Cref{4: Lem main}), we obtain a finite modal space $\F'$ and a stable map $f': \X \surj_\emp \F'$ such that $f'$ is a p-morphism for $x \in \X$ with $\r{f'(x)} < \omega$. Since $L$ is strictly stable, the class $\V(L)\si$ is finitely stable. The dual algebra of $\X$ is in $\V(L)\si$ because $\X$ is topo-rooted and $\X \models L$. So, the dual algebra of $\F'$ is in $\V(L)\si$ by the dual of $\X \surj_\emp \F'$, hence $\F' \models L$.

    Suppose for a contradiction that $\F' \not\models \epsilon(\A_i, D_i)$ for some $i \in I$. Let $\F_i$ be the dual space of $\A_i$ and $\D_i = \sigma[D_i]$. Then, by the dual of \Cref{3: Thm splitting fml char}, there is a topo-rooted closed upset $\F'' \subseteq \F'$ such that $\F'' \surj_{\D_i} \F_i$. Since $\F'$ is finite, $\F''$ is also finite, so it is rooted. Since $\A_i$ is of finite height, $\F_i$ is cycle-free by \Cref{4: Prop cf dual to fh}, and so is $\F''$ by the dual of \Cref{4: Lem stable subalgebra reflects fh}. It follows from \Cref{4: Prop cyclefree char} that there is some $N \in \omega$ such that any point in $\F''$ has rank $\leq N$. Let $s_0$ be the root of $\F''$ and $x_0 \in f'^{-1}(s_0)$. Then $\r{x_0} \leq \r{s_0} \leq N$ by \Cref{4: Lem stable map rank non-decrease}. Let $X' = R^{< \omega}[x_0] = R^{\leq N}[x_0]$ and $\X'$ be $X'$ with the topology and relation induced by $\X$. Then $\X'$ is a closed upset of $\X$ and $x_0$ is a root of $\X'$. Moreover, since every point in $\X'$ has rank $\leq N$, $\X'$ is a $\Kff{N+1}{1}$-space, so $x_0$ is a topo-root of $\X'$ by \Cref{4: Lem K4m1 toporoot}. Since $\r{s_0} < \omega$, we have $f'^{-1}(s_0) \cap \X' = \{x_0\}$, so $\{x_0\}$ is a clopen set in $\X'$, and it follows that $\X'$ is topo-rooted.
    \begin{center}
        \adjustbox{scale=1,center}{
\begin{tikzcd}
\X \arrow[rr, "f'", two heads]                                                                                     &  & \F'                                               &  &      \\
                                                                                                                   &  &                                                   &  &      \\
\X' \arrow[uu, hook] \arrow[rr, "f' {\upharpoonright} \X'"', two heads] \arrow[rrrr, "h'"', two heads, bend right] &  & \F'' \arrow[uu, hook] \arrow[rr, "h"', two heads] &  & \F_i
\end{tikzcd}
        }
    \end{center}

    Recall that $f': \X \surj_\emp \F'$ is surjective and stable. Since $f': \X \surj_\emp \F'$ is a p-morphism for $x \in \X'$ with $\r{f'(x)} < \omega$ and every point in $\F''$ has rank $\leq N$, the restricted map $f' {\upharpoonright} \X': \X' \surj \F''$ is a surjective p-morphism. Let $h:\F'' \surj_{\D_i} \F_i$ and $h' = h \circ f' {\upharpoonright} \X'$. Then, we have $h': \X' \surj_{\D_i} \F_i$ by composition. Thus, $\X'$ is a topo-rooted closed upset of $\X$ such that $\X' \surj_{\D_i} \F_i$. By the dual of \Cref{3: Thm splitting fml char}, this means $\X \not\models \epsilon(\A_i, D_i)$, which contradicts $\X \models L'$. Thus, $\F' \models \epsilon(\A_i, D_i)$ for all $i \in I$, and therefore $\F' \models L'$.
\end{proof}

\begin{theorem} \label{4: Thm fmp ss logic + splitting fml}
    Let $L$ be a strictly stable logic. If $\{\epsilon(\A_i, D_i): i \in I\}$ is a set of stable canonical formulas, then the logic
    \[L + \{\epsilon(\A_i, D_i): i \in I\}\]
    has the fmp.
\end{theorem}

\begin{proof}
    This follows from \Cref{4: Thm main scheme} with \Cref{4: Lem ss logic splitting preserved} and the fact that $\K$ admits (definable) filtration.
\end{proof}

This immediately implies the following fmp result by Blok \cite{blok1978degree}.

\begin{corollary} \label{4: Cor K-union-splittings have fmp}
    Every union-splitting in $\NExt{\K}$ has the fmp.
\end{corollary}

\begin{proof}
    $\K$ is clearly a strictly stable logic. Thus, the statement follows from \Cref{4: Thm K splitting and Jankov formula,4: Thm fmp ss logic + splitting fml}.
\end{proof}

Similar to the case of rule systems, we can show that, for a strictly stable logic $L$, varying the closed domain condition and using finite s.i.~modal algebras outside $\V(L)$ do not add more axiomatization power, as long as the algebras are of finite height. The logics mentioned in \Cref{4: Thm fmp ss logic + splitting fml} are precisely those satisfying the condition (2) in the following theorem, and they form a subclass of union-splittings of the lattice $\NExt{L}$. Here, Jankov formulas are stable canonical formulas of the form $\epsilon(\A, A)$.

\begin{theorem} \label{4: Thm characterization of ss splitting logics}
    Let $L$ be a strictly stable logic. For any logic $L' \in \NExt{L}$, the following are equivalent:
    \begin{enumerate}
        \item[(1)] $L'$ is axiomatized over $L$ by Jankov formulas of finite s.i.~$L$-algebras of finite height,
        \item[(2)] $L'$ is axiomatized over $L$ by stable canonical formulas of finite s.i.~modal algebras of finite height.
    \end{enumerate}
    Moreover, any logic $L' \in \NExt{L}$ satisfying (one of) the above conditions is a union-splitting in $\NExt{L}$.
\end{theorem}

\begin{proof}
    The direction $(1) \Rightarrow (2)$ is clear. For the converse, let 
    \[L' = L + \{\epsilon(\A_i, D_i): i \in I\}\]
    where each $\A_i$ is a finite s.i.~modal algebra of finite height. Let 
    \begin{align*}
        L'' &= L + \{\epsilon(\A, A): \A \text{ is a finite s.i.~$L$-algebra of finite height such that } \\
        & \text{\hspace{250pt}} \A_i \inj_{D_i} \A \text{ for some $i \in I$}\}.
    \end{align*}
    It suffices to show that $L' = L''$. Let $\B$ be an $L$-algebra and $\X$ be its dual space. 

    Suppose that $\B \not\models L''$. Then $\B \not\models \epsilon(\A, A)$ for some finite s.i.~$L$-algebra $\A$ of finite height such that $\A_i \inj_{D_i} \A$. By \Cref{3: Thm splitting fml char}, $\A$ is a subalgebra of a s.i.~homomorphic image $\B'$ of $\B$. Since $\A_i \inj_{D_i} \A$, we have $\A_i \inj_{D_i} \B'$ by composition, which implies $\B \not\models \epsilon(\A_i, D_i)$, hence $\B \not\models L'$.

    Conversely, suppose that $\B \not\models L'$. Then $\B \not\models \epsilon(\A_i, D_i)$ for some $i \in I$, and $\A_i \inj_{D_i} \B'$ for a s.i.~homomorphic image $\B'$ of $\B$ by \Cref{3: Thm splitting fml char}. Let $\X'$ and $\F_i$ be the dual space of $\B'$ and $\A_i$ respectively, and $\D_i = \sigma[D_i]$. Then $\X'$ is a topo-rooted closed upset of $\X$ and there is a stable map $f: \X' \surj_{\D_i} \F_i$. Since $\A_i$ is of finite height, so is $\B'$ by \Cref{4: Lem stable subalgebra reflects fh}, so $\B' \models \Kff{m+1}{1}$ for some $m \in \omega$. Since $\X'$ is topo-rooted, $\X'$ has a topo-root, which is also a root by \Cref{4: Lem K4m1 toporoot}, so $\X'$ is rooted. Applying the Subdivision Construction (\Cref{4: Lem main}) to $f: \X' \surj_{\D_i} \F_i$, we obtain a finite modal space $\F'$, a stable map $f': \X' \surj_\emp \F'$ such that $f'$ is a p-morphism for $x \in \X'$ with $\r{f'(x)} < \omega$, and a stable map $g: \F' \surj_{\D_i} \F_i$ such that $f = g \circ f'$. Let $\A'$ be the dual algebra of $\F'$. Then $\A_i \inj_{D_i} \A'$. Since $\A_i$ is of finite height, so is $\A'$ by \Cref{4: Lem stable subalgebra reflects fh}. So, $\F'$ is cycle-free by \Cref{4: Prop cf dual to fh}. Thus, $f': \X' \surj \F'$ is a p-morphism. Since $\X'$ is rooted and $f'$ is surjective and stable, $\F'$ is also rooted. Also, dually, we have that $\A'$ is a subalgebra of $\B'$. Since $\B'$ is a homomorphic image of $\B$ and $\B \models L$, we obtain $\A' \models L$. Then, $\A'$ is a finite s.i.~$L$-algebra of finite height such that $\A_i \inj_{D_i} \A'$. Thus, $\epsilon(\A', A') \in L''$ by definition. Moreover, since $\A' \inj \B'$ and $\B \surj \B'$, we have $\B \not\models \epsilon(\A', A')$ by \Cref{3: Thm splitting fml char}. So, $\B \not\models L''$, and therefore, $L' = L''$.

    Finally, we show that if $\A$ is a finite s.i.~$L$-algebra (of finite height), then $(L + \epsilon(\A, A), \Log(\A))$ is a splitting pair in $\NExt{L}$. This implies that a logic $L'$ satisfying (1) is a union-splitting in $\NExt{L}$. For any logic $L'' \in \NExt{L}$ such that $L + \epsilon(\A, A) \not\subseteq L''$, there is a modal algebra $\B$ such that $\B \models L''$ and $\B \not\models \epsilon(\A, A)$. So, $\A$ is a subalgebra of a s.i.~homomorphic image of $\B$, which implies $L'' \subseteq \Log(\B) \subseteq \Log(\A)$. Thus, $(L + \epsilon(\A, A), \Log(\A))$ is a splitting pair in $\NExt{L}$.
\end{proof}

If $L = \K$, then these logics are exactly the union-splittings of $\NExt{\K}$ by \Cref{4: Thm K splitting and Jankov formula}. A characterization of union-splittings of $\NExt{\K}$ similar to \Cref{4: Thm characterization of splitting rule system} is formulated and used in \cite{takahashi2025decidabilityunionsplitting} to show that the property of being a union-splitting is decidable in $\NExt{\K}$. However, note that for $L = \logic{KD}$ or $L = \logic{KT}$, there are no non-trivial logics obtained this way because there is no $\logic{KD}$-algebra or $\logic{KT}$-algebra of finite height. We leave it open if there is a strictly stable logic other than $\K$ such that \Cref{4: Thm fmp ss logic + splitting fml} yields non-trivial fmp results.

\subsection{Union-splittings of the lattice of extensions of pre-transitive stable logics}

Finally, we address pre-transitive logics. Recall that pre-transitive logics $\Kff{m}{n}$ are logics axiomatized by $\Dia^m p \to \Dia^n p$ (or equivalently, $\Box^n p \to \Box^m p$) over $\K$. The logic $\Kff{m}{n}$ defines the condition 
\[\forall x \forall y \: (x R^m y \to x R^n y),\]
where $R^k$ is the $k$-time composition of $R$, on modal spaces. Note that $\Kff{2}{1}$ is the transitive logic $\Kf$. We will be interested particularly in the pre-transitive logics $\Kff{m+1}{1}$ $(m \geq 1)$. Gabbay \cite{ageneralfiltrationmethod1972} constructed a definable filtration for these logics and thus proved the fmp of them. An algebraic presentation is provided in \cite{takahashi2025stablecanonicalrulesformulas}, through which the fmp of $\Kff{m+1}{1}$-stable logics is proved. We generalize these results further to a class of union-splittings in the lattice $\NExt{\Kff{m+1}{1}}$.

The following property of $\Kff{m+1}{1}$-stable logics follows from a straightforward generalization of a series of characterizations of $\Kf$-stable logics obtained in \cite{stablemodallogic} to $\Kff{m+1}{1}$-stable logics. Below, we provide a direct proof for completeness. 

\begin{lemma} \label{4: Lem finitely K4m1 stable}
    Let $m \geq 1$. If $L$ is a $\Kff{m+1}{1}$-stable logic, then the class $\V(L)\si$ is finitely $\Kff{m+1}{1}$-stable. 
\end{lemma}

\begin{proof}
    Let $L$ be a $\Kff{m+1}{1}$-stable logic. Then $\V(L)$ is generated by a stable class $\class{K}$ of $\Kff{m+1}{1}$-algebras. We first show that $\V(L)$ is generated by the class $\class{K}\fin$ of finite members of $\class{K}$. It is clear that $\V(\class{K}\fin) \subseteq \V(\class{K}) = \V(L)$. For the converse, let $\phi \notin L$. Then $\frak{D} \not\models \phi$ for some $\frak{D} \in \class{K}$. Since $\Kff{m+1}{1}$ admits definable filtration \cite{ageneralfiltrationmethod1972} (see also \cite{takahashi2025stablecanonicalrulesformulas}), there is a definable filtration $\frak{D}'$ of $\frak{D}$ such that $\frak{D}' \models \Kff{m+1}{1}$ and $\frak{D}' \not\models \phi$. Then, $\frak{D}' \in \class{K}\fin$ since $\frak{D}'$ is finite and $\frak{D}' \inj_\emp \frak{D}$ and $\class{K}$ is stable. So, $\V(\class{K}\fin) \not\models \phi$. Thus, $\V(\class{K}\fin) = \V(L)$. Note that $\class{K}\fin$ is also stable. 
    
    Let $\B \in \V(L)\si$ and $\A$ be a finite $\Kff{m+1}{1}$-algebra such that $\A \inj_\emp \B$. Let $\X = (X, R)$ and $\F = (F, Q)$ be the dual space of $\B$ and $\A$ respectively. Then, there is a stable map $f: \X \surj_\emp \F$. Since $\B$ is s.i., $\X$ is topo-rooted. Let $x_0$ be a topo-root of $\X$. If $s \notin Q^{<\omega}[f(x_0)]$ for some $s \in \F$, then $f^{-1}(s)$ is a non-empty clopen subset of $\X$ such that $f^{-1}(s) \cap R^{<\omega}[x_0] = \emp$, so the closure of $R^{<\omega}[x_0]$ is contained in $\X {\setminus} f^{-1}(s)$, which contradicts $x_0$ being a topo-root of $\X$. Thus, $s \in Q^{<\omega}[f(x_0)]$ for all $s \in \F$, namely, $f(x_0)$ is a root of $\F$. So, $\F$ is rooted, and $\A$ is s.i. since it is finite. 

    Since $\A \inj_\emp \B$ and $\B$ is s.i., $\B \not\models \gamma^m(\A, \emp)$ by \Cref{3: Thm scf char}, so $\gamma^m(\A, \emp) \notin L$ by $\B \models L$. Then, since $\V(L)$ is generated by $\class{K}\fin$, there is some $\C \in \class{K}\fin$ such that $\C \not\models \gamma^m(\A, \emp)$. So, there is some s.i.~homomorphic image $\C'$ of $\C$ such that $\A \inj_\emp \C'$ by \Cref{3: Thm scf char}. Let $\G$ and $\G'$ be the dual space of $\C$ and $\C'$ respectively. Then, $\G'$ is a closed upset of $\G$ and there is a stable map $f: \G' \surj_\emp \F$. If $\G' = \G$, then $\C' = \C$, so $\A \in \class{K}\fin$ since $\class{K}\fin$ is stable, and hence $\A \in \V(L)\si$. Suppose that $\G' \neq \G$. 

    \begin{center}
        \adjustbox{scale=1,center}{
\begin{tikzcd}
\A' \arrow[d, two heads] \arrow[r, hook] & \C \arrow[d, two heads] &  & \G \arrow[r, "f'", two heads]                  & \F'                \\
\A \arrow[r, hook]                       & \C'                     &  & \G' \arrow[u, hook] \arrow[r, "f"', two heads] & \F \arrow[u, hook]
\end{tikzcd}
        }
    \end{center}

    Let $\F'$ be the modal space obtained by adding a fresh point $s_0$ to $\F$ such that $s_0$ sees every point in $\F'$. It is easy to verify that $\F'$ is a $\Kff{m+1}{1}$-space and $\F$ is a closed upset of $\F'$. We define $f': \G \to \F'$ by $f'(t) = f(t)$ if $t \in \G'$ and $f'(t) = s_0$ otherwise. Then, $f'$ is surjective because $f$ is surjective and $\G' \neq \G$. Moreover, $f'$ is stable because $f$ is stable and $s_0$ sees every point in $\F'$. Thus, since a map between two finite modal spaces is always continuous, we have $f': \G \surj_\emp \F'$. Let $\A'$ be the dual algebra of $\F'$. Then, $\A'$ is a $\Kff{m+1}{1}$-algebra such that $\A' \inj_\emp \C$, and $\A$ is a homomorphic image of $\A'$. This implies $\A' \in \class{K}\fin$ and $\A \in \V(\class{K}\fin) = \V(L)$. So, we conclude that the class $\V(L)\si$ is finitely $\Kff{m+1}{1}$-stable. 
\end{proof}

\begin{lemma} \label{4: Lem K4m1 preserve}
    Let $m \geq 1$ and $L$ be a $\Kff{m+1}{1}$-stable logic. If $\{\gamma^m(\A_i, D_i): i \in I\}$ is a set of stable canonical formulas where each $\A_i$ is a finite s.i.~$\Kff{m+1}{1}$-algebra of finite height, then the logic
    \[L + \{\gamma^m(\A_i, D_i): i \in I\}\]
    is preserved by the Subdivision Construction over $\Kff{m+1}{1}$.
\end{lemma}

\begin{proof}
    Let $L' = L + \{\gamma^m(\A_i, D_i): i \in I\}$. Let $\X = (X, R)$ be a topo-rooted $L'$-space, $\F = (F, Q)$ be a finite $\Kff{m+1}{1}$-space, and $f: \X \surj_\D \F$ be a stable map. Applying the Subdivision Construction (\Cref{4: Lem main}), we obtain a finite modal space $\F' = (F', Q')$, a stable map $f': \X \surj_\emp \F'$ such that $f'$ is a p-morphism for $x \in \X$ with $\r{f'(x)} < \omega$,  and a stable map $g: \F' \surj_\D \F$ such that $g$ is an isomorphism between the rank $\omega$ part of $\F'$ and that of $\F$, and $f = g \circ f'$. 

    \begin{claim} \label{4: Claim 1}
        $\F' \models \Kff{m+1}{1}$.
    \end{claim}

    \begin{proof}
        Let $s_0, \dots, s_{m+1} \in \F'$ be such that $s_iQ's_{i+1}$ for each $0 \leq i \leq m$. We show that $s_0Q's_{m+1}$. We divide cases by $\r{s_0}$.
        
        If $\r{s_0} < \omega$, then $\r{s_i} < \omega$ for all $0 \leq i \leq m+1$. Since $f': \X \surj_\emp \F'$ is surjective and is a p-morphism for $x \in \X$ with $\r{f'(x)} < \omega$, we obtain a chain $x_0R \cdots R x_{m+1}$ in $\X$ such that $f'(x_i) = s_i$ for each $0 \leq i \leq m+1$. Then $x_0 R x_{m+1}$ since $\X \models \Kff{m+1}{1}$ by assumption, which implies $s_0 Q' s_{m+1}$ since $f'$ is stable. 

        Suppose that $\r{s_0} = \omega$. Since $g: \F' \surj_\D \F$ is stable, we have $g(s_0) Q \cdots Q g(s_{m+1})$ in $\F$. Then $g(s_0) Q g(s_{m+1})$ since $\F \models \Kff{m+1}{1}$ by assumption. It follows from the Subdivision Construction that $s_0 = \<g(s_0), S\>$ and $s_{m+1} = \<g(s_{m+1}), S'\>$ for some $S, S' \subseteq F'$. Also, since $\r{s_0} = \omega$, we have $\r{g(s_0)} = \omega$ by \Cref{4: Lem stable map rank non-decrease} (thus, in fact, $S = \emp$). So, $s_0Q's_{m+1}$ by $g(s_0) Q g(s_{m+1})$ and the definition of $Q'$.
        
        Thus, in both cases we have $s_0Q's_{m+1}$. Hence, $\F' \models \Kff{m+1}{1}$.
    \end{proof}

    The class $\V(L)\si$ is finitely $\Kff{m+1}{1}$-stable by \Cref{4: Lem finitely K4m1 stable}. Since the dual algebra of $\X$ is in $\V(L)\si$, the dual algebra of $\F'$ is a finite $\Kff{m+1}{1}$-algebra by \Cref{4: Claim 1}, and the latter is a stable subalgebra of the former by $\X \surj_\emp \F'$, we see that the dual algebra of $\F'$ is in $\V(L)\si$. In particular, we have $\F' \models L$.

    Suppose for a contradiction that $\F' \not\models \gamma^m(\A_i, D_i)$ for some $i \in I$. Let $\F_i$ be the dual space of $\A_i$ and $\D_i = \sigma[D_i]$. Then, by the dual of \Cref{3: Thm scf char}, there is a topo-rooted closed upset $\F'' \subseteq \F'$ such that $\F'' \surj_{\D_i} \F_i$. A similar argument as in the proof of \Cref{4: Lem ss logic splitting preserved} yields a topo-rooted closed upset $\X'$ of $\X$ such that $\X' \surj_{\D_i} \F_i$. By the dual of \Cref{3: Thm scf char}, this means $\X \not\models \gamma^m(\A_i, D_i)$, which contradicts $\X \models L'$. Thus, $\F' \models \gamma^m(\A_i, D_i)$ for all $i \in I$, and therefore $\F' \models L'$.
\end{proof}

\begin{theorem} \label{4: Thm fmp K4m1 stable + sp}
    Let $m \geq 1$ and $L$ be a $\Kff{m+1}{1}$-stable logic. If $\{\gamma^m(\A_i, D_i): i \in I\}$ is a set of stable canonical formulas where each $\A_i$ is a finite s.i.~$\Kff{m+1}{1}$-algebra of finite height, then the logic
    \[L + \{\gamma^m(\A_i, D_i): i \in I\}\]
    has the fmp.
\end{theorem}

\begin{proof}
    This follows from \Cref{4: Thm main scheme} with \Cref{4: Lem K4m1 preserve} and the fact that $\Kff{m+1}{1}$ admits definable filtration \cite{ageneralfiltrationmethod1972} (see also \cite{takahashi2025stablecanonicalrulesformulas}).
\end{proof}

Similar to the previous two cases, we again obtain the result that, for a $\Kff{m+1}{1}$-stable logic $L$, varying the closed domain condition and using finite s.i.~$\Kff{m+1}{1}$-algebras outside $\V(L)$ does not add more axiomatization power, as long as the algebras are of finite height. The logics mentioned in \Cref{4: Thm fmp K4m1 stable + sp} are precisely those satisfying the condition (2) in the following theorem, and they form a subclass of union-splittings of the lattice $\NExt{L}$. Here, Jankov formulas are stable canonical formulas of the form $\gamma^m(\A, A)$.

\begin{theorem} \label{4: Thm characterization of K4m1 splitting logics}
    Let $m \geq 1$ and $L$ be a $\Kff{m+1}{1}$-stable logic. For any logic $L' \in \NExt{L}$, the following are equivalent:
    \begin{enumerate}
        \item[(1)] $L'$ is axiomatized over $L$ by Jankov formulas of finite s.i.~$L$-algebras of finite height,
        \item[(2)] $L'$ is axiomatized over $L$ by stable canonical formulas of finite s.i.~$\Kff{m+1}{1}$-algebras of finite height.
    \end{enumerate}
    Moreover, any logic $L' \in \NExt{L}$ satisfying (one of) the above conditions is a union-splitting in $\NExt{L}$.
\end{theorem}

\begin{proof}
    If $\A$ is a finite s.i.~$\Kff{m+1}{1}$-algebra of finite height, it follows from \Cref{3: Thm scf char} and \Cref{3: Thm splitting fml char} that Jankov formulas as a special type of stable canonical formulas of the form $\gamma^m(\A, A)$ and of the form $\epsilon(\A, A)$ have the same semantic characterization for $\Kff{m+1}{1}$-algebras. Moreover, as $\Kff{m+1}{1}$-stable logics have the fmp \cite{takahashi2025stablecanonicalrulesformulas}, it is a straightforward generalization of \Cref{4: Thm K4m1 stable logic and stable formula} that, for a $\Kff{m+1}{1}$-stable logic $L$ and a logic $L' \in \NExt{L}$, $L'$ is a union-splitting of $\NExt{L}$ iff $L'$ is axiomatizable over $L$ by Jankov formulas. So, this theorem is proved by a similar argument as the proof of \Cref{4: Thm characterization of ss splitting logics}. 
\end{proof}

\begin{remark}
    Contrary to the case of rule systems (\Cref{4: Thm characterization of splitting rule system}), this subclass of union-splittings in $\NExt{\Kff{m+1}{1}}$ is not a lower part of them. We provide a counterexample for $L = \Kf$. Let $\F$ be the finite modal space $\refseeirref$ and $\F'$ be its upset $\irrefsingle$. Let $\A$ and $\A'$ be the dual algebras of $\F$ and $\F'$ respectively. Then, $\A'$ is of finite height since $\F'$ is cycle-free, so $\Kf + \gamma(\A', A')$ is a union-splitting in the subclass (in fact, it is the logic $\logic{K4D}$). It is clear that $\Log(\F) \subsetneq \Log(\F')$, so $\Kf + \gamma(\A, A)$ is a union-splitting such that $\Kf + \gamma(\A, A) \subsetneq \Kf + \gamma(\A', A')$. Suppose for a contradiction that $\Kf + \gamma(\A, A)$ is in the subclass. Then, $\Kf + \gamma(\A, A) = \Kf + \{\gamma(\A_i, A_i): i \in I\}$ for a set $\{\gamma(\A_i, A_i): i \in I\}$ of Jankov formulas where each $\A_i$ is a finite s.i.~$\Kf$-algebra of finite height. Since $\A \not\models \gamma(\A, A)$, there is some $i \in I$ such that $\A \not\models \gamma(\A_i, A_i)$, so $\A_i \inj \B$ for a s.i.~homomorphic image of $\A$. By \Cref{4: Lem stable subalgebra reflects fh}, $\B$ is of finite height. However, up to isomorphism, since $\F'$ is the only cycle-free upset of $\F$, we have $\B = \A'$, and also $\A_i = \A'$ because $\F'$ is the only p-morphic image of itself. Thus, we obtain 
    \[\Kf + \gamma(\A, A) \subsetneq \Kf + \gamma(\A', A') \subseteq \Kf + \{\gamma(\A_i, A_i): i \in I\} = \Kf + \gamma(\A, A),\]
    which is a contradiction. So, the subclass of the union-splittings of $\NExt{\Kf}$ mentioned in \Cref{4: Thm fmp K4m1 stable + sp,4: Thm characterization of K4m1 splitting logics} is not downward closed. 
\end{remark}

From the global viewpoint of the lattice $\NExt{\Kff{m+1}{1}}$, our fmp results also have implications on the \emph{degree of Kripke incompleteness}, initially introduced by Fine \cite{fineIncompleteLogicContainingS41974} (see also \cite[Section 10.5]{czModalLogic1997}). 

\begin{definition} \label{2: Def degree of ki}
    Let $L_0$ be a logic. For a logic $L \in \NExt{L_0}$, the \emph{degree of Kripke incompleteness} of $L$ in $\NExt{L_0}$ is the cardinal 
    \[|\{L' \in \NExt{L_0}:  \KF(L') = \KF(L)\}|,\]
    where $\KF(L)$ is the class of Kripke frames validating $L$.
\end{definition}

Intuitively, the degree of Kripke incompleteness of a logic $L$ measures the number of logics that cannot be distinguished from $L$ by Kripke frames. A logic that has the degree of Kripke incompleteness 1 is also said to be \emph{strictly Kripke complete} in \cite[Section 10.5]{czModalLogic1997}. A celebrated result of Blok \cite{blok1978degree} says that, in $\NExt{\K}$, a logic $L$ has degree 1 iff $L$ is a union-splitting, and $L$ has degree $\conti$ otherwise. Determining the degree of Kripke incompleteness in lattices other than $\NExt{\K}$ (e.g., $\NExt{\Kf}$) is a long-standing open question (e.g., \cite[Problem 10.5]{czModalLogic1997}). It is even unknown whether all union-splittings in $\NExt{\Kf}$ are Kripke complete, let alone what their degree of Kripke incompleteness is. Arguably the best result so far is that logics of finite depth have degree 1 in $\NExt{\Kf}$, where a logic is of \emph{finite depth} if any Kripke frame validating it contains no chain of length more than $n$ for some $n \in \omega$. This result follows from the fact that transitive logics of finite depth are union-splittings in $\NExt{\Kf}$ (see, e.g., \cite{FrameBasedFormulas2008}) and have the fmp \cite{Segerberg1971}. In this respect, \Cref{4: Thm fmp K4m1 stable + sp} identifies a class of union-splittings in $\NExt{\Kf}$ that have the fmp and thus are Kripke complete, and hence have degree of Kripke incompleteness 1 in $\NExt{\Kf}$. Moreover, this class is disjoint from the class of transitive logics of finite depth.

\begin{corollary} \label{4: Cor degree 1}
    Let $m \geq 1$. Every union-splitting in $\NExt{\Kff{m+1}{1}}$ split by a set of finite s.i.~$\Kff{m+1}{1}$-algebras of finite height has the degree of Kripke incompleteness 1 in $\NExt{\Kff{m+1}{1}}$.
\end{corollary}

\begin{proof}
    This follows directly from \Cref{4: Thm fmp K4m1 stable + sp} and the general fact that a Kripke complete union-splitting $L$ in $\NExt{L_0}$ has degree of Kripke incompleteness 1 in $\NExt{L_0}$ \cite{blok1978degree} (see also \cite[Section 10.5]{czModalLogic1997}). 
\end{proof}

\begin{proposition}
    Let $m \geq 1$. Every union-splitting in $\NExt{\Kff{m+1}{1}}$ split by a set of finite s.i.~$\Kff{m+1}{1}$-algebras of finite height is not of finite depth.
\end{proposition}

\begin{proof}
    Let $L = \Kff{m+1}{1} + \{\gamma^m(\A_i, A_i): i \in I\}$ where each $\A_i$ is a finite s.i.~$\Kff{m+1}{1}$-algebra of finite height. Let $\F$ be the Kripke frame $(\omega, <)$. Then, $\F \models \Kff{m+1}{1}$ since $\F$ is transitive. Suppose for a contradiction that $\F \not\models \gamma^m(\A_i, A_i)$ for some $i \in I$. Let $\F_i$ be the dual space of $\A_i$, so it is cycle-free by \Cref{4: Prop cf dual to fh}. By an argument similar to the proof of \Cref{3: Thm scf char} (see, e.g., \cite{stablecanonicalrules,takahashi2025stablecanonicalrulesformulas}), $\F_i$ is a p-morphic image of an upset (i.e., generated subframe) of $\F$. Since any upset of $\F$ is isomorphic to $\F$, an infinite chain, this contradicts that $\F_i$ is cycle-free. Thus, $\F \models \gamma^m(\A_i, A_i)$ for all $i \in I$, and hence $\F \models L$. Therefore, $\F$ witnesses that $L$ is not of finite depth.
\end{proof}

\section{Conclusions and future work} \label{Sec 5}

We have applied the Subdivision Construction to prove the fmp for three classes of logics and rule systems. In all three cases, the key step was to show the corresponding preservation lemma under the Subdivision Construction so that the fmp follows automatically from the general scheme \Cref{4: Thm main scheme}. Further applications of this method remain open. We expect there are more logics and rule systems that are preserved by the construction so that we can use \Cref{4: Thm main scheme} to show their fmp. A characterization, or more modestly, a sufficient condition, of logics or rule systems preserved by the Subdivision Construction will lead to more general fmp results. 

The statement of \Cref{4: Lem main} may remind one of factorization systems in category theory, but a precise categorical formulation remains open. Even stable homomorphisms and the closed domain condition lack a categorical formalization. Although compositions of stable homomorphisms are stable homomorphisms and compositions of homomorphisms are homomorphisms, it is not clear for which closed domain the composition of a stable homomorphism satisfying CDC for $D$ and one satisfying CDC for $D'$ satisfies CDC. It would also be insightful to explore an algebraic presentation of the Subdivision Construction.

Finally, as discussed in the introduction, the series of fmp results presented in this paper can be compared with the general fmp result proved by Zakharyaschev \cite{ZakharyaschevCanonical3} regarding cofinal subframe logics and canonical formulas (see also \cite[Theorem 11.55]{czModalLogic1997}). Despite the similarity in appearance, Zakharyaschev's results have an interesting consequence that every extension of $\Sf$ (or $\IPC$) with finitely many axioms in one variable has the fmp \cite[Theorem 11.58 and Corollary 11.59]{czModalLogic1997}. It remains open whether our results lead to such an elegant fmp result.

\vspace{\baselineskip}
\begin{acknowledgements}
I am very grateful to Nick Bezhanishvili for his supervision of the Master's thesis and his valuable comments on this paper. I would also like to thank Rodrigo Nicolau Almeida and Qian Chen for insightful discussions and feedback. I was supported by the Student Exchange Support Program (Graduate Scholarship for Degree Seeking Students) of the Japan Student Services Organization and the Student Award Scholarship of the Foundation for Dietary Scientific Research. 
    
\end{acknowledgements}

\printbibliography[heading=bibintoc,title=References]

\end{document}